\documentclass[11pt,reqno]{amsart}
\usepackage{amsmath, amsthm, amssymb, epsfig, amsfonts, bbm, euscript, graphicx, graphics, url}
\usepackage{pb-diagram}
\usepackage{lamsarrow}
\usepackage[colorlinks]{hyperref}

\newtheorem{theorem}{Theorem}[section]
\newtheorem{lemma}[theorem]{Lemma}
\newtheorem{proposition}[theorem]{Proposition}
\newtheorem{corollary}[theorem]{Corollary}

\newtheorem{question}[theorem]{Question}
\newtheorem{conjecture}[theorem]{Conjecture}

\theoremstyle{definition}
\newtheorem{definition}[theorem]{Definition}
\newtheorem{example}[theorem]{Example}

\newtheorem{remark}[theorem]{Remark}
\newtheorem*{acknowledgements}{Acknowledgements}

\numberwithin{equation}{section}

\newskip\aline \newskip\halfaline
\DeclareMathOperator{\Vol}{Vol}
\DeclareMathOperator{\Symp}{Symp}

\DeclareMathOperator{\rank}{rank}

\newcommand{\bC}{\mathbb C}

\newcommand{\bP}{\mathbb P}

\newcommand{\bR}{\mathbb R}

\newcommand{\bZ}{\mathbb Z}

\newcommand{\ba}{\textbf{a}}
\newcommand{\bb}{\textbf{b}}
\newcommand{\bc}{\textbf{c}}

\title{Toric Structures on Symplectic Bundles of Projective Spaces}
\author{Andrew Fanoe}
\begin{document}

\begin{abstract}Recently, extending work by Karshon, Kessler and Pinsonnault \cite{KKP}, Borisov and McDuff showed in \cite{D} that a given symplectic manifold $(M,\omega)$ has a finite number of distinct toric structures.  Moreover, in \cite{D} McDuff also showed a product of two projective spaces $\bC P^r\times \bC P^s$ with any given symplectic form has a unique toric structure provided that $r,s\geq 2$. In contrast, the product $\bC P^r \times \bC P^1$ can be given infinitely many distinct toric structures, though only a finite number of these are compatible with each given symplectic form $\omega$.  In this paper we extend these results by considering the possible toric structures on a toric symplectic manifold $(M,\omega)$ with $\dim H^2(M)=2$.  In particular, all such manifolds are $\bC P^r$ bundles over $\bC P^s$ for some $r,s$.  We show that there is a unique toric structure if $r<s$, and also that if $r,s\geq 2$ then $M$ has at most finitely many distinct toric structures that are compatible with any symplectic structure on $M$.  Thus, in this case the finiteness result does not depend on fixing the symplectic structure.  We will also give other examples where $(M,\omega)$ has a unique toric structure, such as the case where $(M,\omega)$ is monotone.
\end{abstract}

\maketitle

\section{Introduction}

Let $\widetilde{M}\rightarrow M\rightarrow\widehat{M}$ be a locally trivial fibration.  Then M is a symplectic bundle if $\widetilde{M}$ has a symplectic structure $\omega_0$ and if the structure group of the bundle is $\Symp(\widetilde{M})$.  By Lemma \ref{toricclassifyinglemma} below, any symplectic toric manifold $M$ with $\dim H^2(M)=2$ is a $\bC P^r$-bundle over $\bC P^s$.  Furthermore, any toric structure can be realized as the projectivization $\bP(L_{-a_0}\oplus L_{-a_1}\oplus \dots \oplus L_{-a_r})$ of a sum of complex line bundles $L_{-a_i}$ over $\bC P^s$ with the obvious action of the torus $T^{r+s}$, where $L_c$ is the line bundle over $\bC P^s$ with first Chern class $c$ times a generator of $H^2(\bC P^s)$. By tensoring with $L_{c}$ where $c=-\max a_i$ we may assume that $a_0=0\leq a_1\leq\ldots\leq a_r$.  Moreover the symplectic form $\omega$ restricts to the standard Fubini-Study form on the fiber, and so may be normalized by requiring that $\omega(\ell)=r+1$, where $\ell$ is the homology class of a line in the fiber.  Since $H^2(M)$ has rank $2$, the above normalization still leaves $[\omega]$ with one free parameter.  We call this parameter $\kappa$, and it can be easily seen to be determined by $\Vol(M,\omega)$, as in Lemma \ref{volumetheorem}. Thus from the above we see that the tuple $(\ba;\kappa):= (a_1,\ldots,a_r;\kappa)$ determines a toric structure on a symplectic toric manifold $(M,\omega)$ with $\dim H^2(M)=2$, where $0\leq a_1\leq\ldots\leq a_r$ and $\kappa$ is related to the symplectic volume of $M$.

We denote the resulting toric manifold $(M_{\ba},\omega_{\ba}^\kappa, T_{\ba})$.  By Definition \ref{bundledefinition} and Lemma \ref{bundlecorrespondencetheorem} below, for each tuple $\ba$ there is a number $K_{\ba}(s)=\sigma_1(\ba)-s$ such that $M_{\ba}$ admits the structure described above for all $\kappa>K_{\ba}(s)$.  Furthermore, we have the following fundamental result, which is proven at the end of section $2$.

\begin{theorem}\label{toricclassifyingtheorem}Let $(M,\omega,T)$ be a toric symplectic manifold with $\dim H^2(M)=2$.  Then there is a unique tuple $(\ba;\kappa)$ with $0\leq a_1\leq\ldots\leq a_r$ such that $(M,\omega,T)$ is equivariantly symplectomorphic to $(M_{\ba},\omega_{\ba}^\kappa,T_{\ba})$.
\end{theorem}

Thus, to count toric structures on manifolds with $\dim H^2(M)=2$, it suffices to count toric structures on the manifolds $M_{\ba}$.

The following result is based on Theorem $6.1$ of \cite{CMS} and is instrumental to the proofs of many of our results.  Due to its important role in the rest of the paper, we will give the details of the proof using our notation in section $2$.

\begin{proposition}\label{mainlemma}  Let $M_{\ba}$ and $M_{\bb}$ be $\bC P^r$ bundles over $\bC P^s$ as above for some vectors $\ba$ and $\bb$.  The following are equivalent:
\begin{enumerate}
 \item $H^*(M_{\ba};\bZ)$ is isomorphic to $H^*(M_{\bb};\bZ)$ as a ring.
 \item $\bP(L_{0}\oplus L_{-a_1}\oplus \dots \oplus L_{-a_r})$ is isomorphic to $\bP(L_{0}\oplus L_{-b_1}\oplus \dots \oplus L_{-b_r})$ as a
 projective vector bundle.
 \item $M_{\ba}$ is isomorphic to $M_{\bb}$ as a symplectic bundle.
 \item There is $C\in \bZ$ such that
 \begin{equation*}
 \sigma_i(C,C+\ba):=\sigma_i(C,a_1+C,\dots, a_r+C) = \sigma_i(0,b_1,\dots, b_r)=:\sigma_i(0,\bb),\quad 1\leq i\leq \min\{r+1,s\}
 \end{equation*}
 where $\sigma_i$ denotes the $i$th elementary symmetric function.
\end{enumerate}
\end{proposition}

It is natural to conjecture that if $(M_{\ba},\omega^\kappa_{\ba})$ is isomorphic to $(M_{\bb},\omega^\kappa_{\bb})$ as a symplectic bundle, i.e. there exists a diffeomorphism $\phi:M_{\ba}\rightarrow M_{\bb}$ preserving the fiberwise symplectic structure, then they are symplectomorphic for all $\kappa>\max (K_{\ba},K_{\bb})$.  However this is not yet known except when $s=1$ or, as in Lemma \ref{deformationequivalentlemma} below, if $\kappa\gg 0$.  In fact, we have the following theorem, proven in section $3$.

\begin{theorem}\label{symplectomorphismtheorem}
Let $(M_{\ba},\omega^\kappa_{\ba})$ and $(M_{\bb},\omega^\kappa_{\bb})$ be $\bC P^r$ bundles over $\bC P^1$ as above with $\kappa>\max(K_{\ba},K_{\bb})$.  Then $(M_{\ba},\omega^\kappa_{\ba})$ is symplectomorphic to $(M_{\bb},\omega^\kappa_{\bb})$ if and only if $(M_{\ba},\omega^\kappa_{\ba})$ is isomorphic to $(M_{\bb},\omega^\kappa_{\bb})$ as a symplectic bundle.
\end{theorem}

Since this is not known in the general case $s>1$, we will consider the following weaker notion of equivalence.

\begin{definition} We say that two symplectic manifolds $(M,\omega), (M',\omega')$ are {\bf deformation equivalent} and write $(M,\omega)\sim (M',\omega')$ if there is a diffeomorphism $\phi:M\rightarrow M'$ and a family $\omega_t, t\in [0,1],$ of symplectic forms on $M$
such that
\[
\phi^*([\omega']) = [\omega],\quad \omega_0=\phi^*(\omega'),\quad \omega_1=\omega.
\]
\end{definition}

\begin{remark}In contrast to the usual definition of deformation equivalence, we have required $\phi^*([\omega'])=[\omega]$. Thus the deformation starts and ends in the same cohomology class, even if it leaves this class for some $t$.
\end{remark}

The following lemma says that $(M_{\ba},\omega_{\ba}^\kappa)$ and $(M_{\bb},\omega_{\bb}^\kappa)$ are isomorphic as symplectic bundles if and only if they are deformation equivalent, and it will be proven in section $3$.

\begin{lemma}\label{deformationequivalentlemma}Let $\ba=(a_1,\ldots,a_r)$, $\bb=(b_1,\ldots,b_r)$ and $\kappa$ determine the bundles $(M_{\ba},\omega_{\ba}^\kappa)$ and $(M_{\bb},\omega_{\bb}^\kappa)$. Then $M_{\ba}$ and $M_{\bb}$ are isomorphic as symplectic bundles if and only if $(M_{\ba},\omega_{\ba}^\kappa)$ is deformation equivalent to $(M_{\bb},\omega_{\bb}^\kappa)$.  Moreover, for $\kappa\gg 0$, we also have $(M_{\ba},\omega_{\ba}^\kappa)$ is symplectomorphic to $(M_{\bb},\omega_{\bb}^\kappa)$.
\end{lemma}

Given the class of manifolds $(M_{\ba},\omega^\kappa_{\ba})$, we can ask how many different toric structures we can put on the same deformation class.  Given symplectic toric manifolds $(M,\omega,T)$ and $(M',\omega',T')$, we recall that the toric structures are called \textbf{equivalent} if there is an equivariant symplectomorphism from one to the other, and are called \textbf{inequivalent} otherwise.

The following result uses the fact that two toric manifolds are equivalent if and only if their moment polytopes are affine equivalent, and is proven in section $2$.

\begin{lemma}\label{toricstructurelemma}
Let $\ba=(a_1,\ldots, a_r)$ and $\bb=(b_1,\ldots,b_r)$ have $0\leq a_1\leq\ldots\leq a_r$ and $0\leq b_1\leq\ldots\leq b_r$ and let $\kappa$ and $\kappa'$ be real numbers.  Then $(M_{\ba},\omega^\kappa_{\ba},T_{\ba})$ is equivalent to $(M_{\bb},\omega^{\kappa'}_{\bb},T_{\bb})$ $\iff(\ba;\kappa)=(\bb;\kappa')$.
\end{lemma}

\begin{question}Given a tuple $(\ba;\kappa)$, what is $N_n(\ba;\kappa)$, the number of inequivalent toric structures on the deformation class of $(M^{2n}_{\ba},\omega_\kappa^{\ba})$?  In particular, for fixed $\ba$ and $n$, how does it depend on $\kappa$, and for which $(\ba;\kappa)$ do we have $N_n(\ba;\kappa)=1$?
\end{question}

If $\ba=0$, the manifold $M_{\ba}$ is just a product $\bC P^r\times\bC P^s$ and this question was answered in \cite{D} as follows.

\begin{proposition}\label{producttheorem}(\cite{D}, Prop $1.8$) Let $(M,\omega)=(\bC P^r\times\bC P^s,\omega_r\times\lambda\omega_s)$.  Then if either $r,s\geq 2$, or $r>s\geq 1$ and $\lambda\leq 1$, or $r=s=1$ and $\lambda=1$, there is a unique toric structure compatible with this symplectic structure.  In all other cases, the toric structure is not unique.
\end{proposition}
In light of this proposition, we will focus on the case where $\ba\neq 0$, and hence assume some $a_i\neq 0$.
We now provide a summary of our results, to be proven in section $4$.  Our results are complete when $r<s$, but there are still some open questions for $r\geq s$.  Recall that $r$ is the dimension of the fiber and $s$ is the dimension of the base.

The first main theorem is a uniqueness result.
\begin{theorem}\label{r<stheorem}Let $(M^{2n},\omega_\kappa^a)$ be determined by $\ba=(a_1,\ldots,a_r)$ and $\kappa$, as before.  If $r<s$, we have
\begin{equation*}
N(\ba;\kappa)= \left\{
\begin{array}{ll}
0 & \text{if } \kappa\leq K_{\ba}(s)\\
1 & \text{if } \kappa> K_{\ba}(s)
\end{array} \right.
\end{equation*}
where $K_{\ba}(s):=\sigma_1(\ba)-s$.
\end{theorem}
This gives a complete characterization for $N_n(\ba;\kappa)$ with $r<s$ and has the following interesting corollary.
\begin{corollary}Let $(M,\omega,T)$ and $(M',\omega',T')$ be toric $\bC P^r$ bundles over $\bC P^s$ with $r<s$.  If $(M,\omega)$ is deformation equivalent to $(M',\omega')$, then $(M,\omega,T)$ is equivariantly symplectomorphic to $(M',\omega',T')$.
\end{corollary}
\begin{proof}This follows directly from the fact that if $r<s$, then $N_n(\ba,\kappa)=1$ for all $\kappa>K_{\ba}(s)$.  Indeed, the formula $N_n(\ba,\kappa)=1$ implies that for a fixed tuple $(\ba;\kappa)$, any manifold $(M,\omega,T)$ so that $(M,\omega)$ is deformation equivalent to $(M_{\ba},\omega_{\ba}^\kappa)$ is equivariantly symplectomorphic to $(M_{\ba},\omega_{\ba}^\kappa,T_{\ba})$.
\end{proof}

The case where $r\geq s$ is more complicated, as the next example shows.
\begin{example}Let $r,s=3,2$ and take $\ba=(1,4,4)$ and $\bb=(2,2,5)$. For sufficiently large $\kappa$, both $M_{\ba}^\kappa$ and $M_{\bb}^\kappa$ describe $\bC P^3$ bundles over $\bC P^2$ which are deformation equivalent, by Proposition \ref{mainlemma}.  However, these are obviously not the same toric manifold by Lemma \ref{toricstructurelemma} since $\ba\neq \bb$.  In fact, we show in Example \ref{r=3s=2example} that
\begin{equation*}
N_5(\ba;\kappa)= \left\{
\begin{array}{ll}
0 & \text{if } \kappa\leq 7\\
2 & \text{if } \kappa>7
\end{array} \right.
\end{equation*}
so that there is no choice of $\kappa$ for which $N_5(\ba;\kappa)=1$.
\end{example}

We will give more specific examples of the $r\geq s$ case at the end of section $4$.

Even though there is not a general uniqueness theorem for the $r\geq s$ case, there are still some uniqueness results.  In particular, by restricting the size of $\kappa$, we have the following theorem.
\begin{theorem}\label{monotoneuniquenesstheorem}Let $\ba$ and $\kappa$ be as before with the added assumption that $\kappa\leq 1$.  We have
\begin{equation*}
N_n(\ba;\kappa)= \left\{
\begin{array}{ll}
0 & \text{if } \kappa\leq K_{\ba}(s)\\
1 & \text{if } K_{\ba}(s)<\kappa\leq 1.
\end{array} \right.
\end{equation*}
\end{theorem}
This has an interesting application.  Recall that we call a symplectic manifold monotone if we have $[\omega]=k[c_1(M)]$, for some positive constant $k$ which is usually normalized to equal $1$.  Our notation is chosen so that we have the following, as in Remark $2.2(i)$ of \cite{D}.
\begin{lemma}
$(M_{\ba},\omega^\kappa_{\ba})$ is monotone $\iff$ $\kappa=1$.
\end{lemma}
In question $1.11$ of \cite{D}, McDuff conjectures that every monotone symplectic toric manifold has a unique toric structure.  An obvious corollary of the above gives some support for this conjecture.
\begin{corollary}If $(M_{\ba}^{2n},\omega_{\ba}^\kappa)$ is monotone, then $N_n(\ba;\kappa)=1$.
\end{corollary}
\begin{remark}\label{Fanodefinition}Recall that a toric symplectic manifold $(M,\omega,T)$ is called Fano if there is a smooth family of $T$-invariant forms with $\omega_0=\omega$ and $\omega_1$ monotone.  In our case, we have that $(M^{2n}_{\ba},\omega_{\ba}^\kappa,T_{\ba})$ is Fano if and only if $K_{\ba}(s)<1$.  We will call such vectors $\ba$ Fano vectors.  In examples \ref{fanoexample1}, \ref{fanoexample2}, and \ref{fanoexample3} below, we will consider the Fano case in some specific examples.
\end{remark}

We cannot yet compute $N(\ba;\kappa)$ for all $\kappa$.  However, we can say that, as a function of $\kappa$, $N(\ba;\kappa)$ is monotonic and locally constant, with the only jumps possible being at certain integer values of $\kappa$. Furthermore, if $r=s$, these jumps are of size at most $1$.  More specifically, we have the following theorem.
\begin{theorem}\label{stepfunctiontheorem}Let $\ba=(a_1,\ldots,a_r)$ and let $K_M$ be the minimum of the set of $\kappa$ with $N_n(\ba;\kappa)\neq 0$.  Then we have the following
\begin{enumerate}
\item Let $l\geq 0$ an integer, and $\kappa_1,\kappa_2$ real numbers.  Then we have
 \[
 K_M+l(r+1)<\kappa_1,\kappa_2\leq K_M+(l+1)(r+1)\Longrightarrow N_n(\ba;\kappa_1)=N_n(\ba;\kappa_2).
 \]
\item If also $r=s$ and $\kappa=K_M+l(r+1)$ and $0<\epsilon\leq r+1$,
 \[
 N_n(\ba;\kappa)\leq N_n(\ba;\kappa+\epsilon)\leq N(\ba;\kappa)+1.
 \]
\end{enumerate}
\end{theorem}
Notice that the number $K_M$ above need not equal the number $K_{\ba}(s)$ from before because there could be a different vector $\bb$ so that $(M_{\ba},\omega^\kappa_{\ba})\sim(M_{\bb},\omega^\kappa_{\bb})$ with $K_{\bb}(s)<K_{\ba}(s)$.  Then, we would have $N_n(\ba;\kappa)=N_n(\bb;\kappa)$ for all $\kappa$, while $K_{\bb}(s)<K_{\ba}(s)$, which would obviously imply that $K_M\leq K_{\bb}(s)<K_{\ba}(s)$.

The above theorem then says that if $K_M$ denotes the position of the first jump of $N_n(\ba;\kappa)$, then all subsequent jumps can only occur at the integers $K_M+l(r+1)$, and when $r=s$, these jumps are of size $0$ or $1$.
An obvious corollary of the above two theorems is another uniqueness result.
\begin{corollary}Let $K_M$ be as before, and assume that $r=s$.  Further assume we have an $\ba=(a_1,\ldots,a_r)$ so that $K_{\ba}=K_M$.  Then we have
\[
N_n(\ba;\kappa)=1,\quad\forall~K_M<\kappa\leq K_M+r+1.
\]
\end{corollary}
The following result describes the behavior of $N_n(a;\kappa)$ for large $\kappa$.
\begin{theorem}\label{finitenesstheorem}Let $\ba=(a_1,\ldots,a_r)$ and $\bb=(b_1,\ldots,b_r)$ be as before and let $C$ be an integer, as in Proposition \ref{mainlemma}.  Furthermore, assume that
\[
\sigma_i(0,\bb)=\sigma_i(C,C+\ba)~i=1,\ldots,s
\]
with $s\geq 2$.  Then we have
\[
-\tfrac{1}{r+1}\sigma_1(\ba)\leq C\leq \tfrac{r-1}{r}\sigma_1(\ba)
\]
Moreover, this implies
\[
\kappa_1,\kappa_2\geq (r+1-\tfrac{1}{r})\sigma_1(\ba)-s \Longrightarrow N_n(\ba;\kappa_1)=N_n(\ba;\kappa_2)
\]
In particular, we have
\[
N_n(\ba;\infty):=\lim_{\kappa\rightarrow\infty}N_n(\ba;\kappa)<\infty
\]
\end{theorem}
\begin{remark}This result is surprising at first glance.  The condition $N_n(\ba;\infty)<\infty$ implies that there are at most finitely many toric structures which are compatible with an arbitrary symplectic structure on the given deformation class of $M_a$.  This is stronger than the finiteness result proven by Borisov and McDuff in \cite{D}, which relies on fixing a symplectic structure to get finiteness.

However, if $r=s=1$, this does not happen.  Indeed, in that case, $\ba=a$ and $\bb=b$ are just numbers, and the manifolds $(M_{\ba},\omega_{\ba}^\kappa)$ are the well known Hirzebruch surfaces.  It is known for the Hirzebruch surfaces that if $b-a$ is even, then $(M_{\ba},\omega_{\ba}^\kappa)\sim(M_{\bb},\omega_{\bb}^\kappa)$, which shows that for any a, we have
\[
\lim_{\kappa\rightarrow\infty}N_2(\ba;\kappa)=\infty
\]
\end{remark}

The above theorem says that for any $\ba$, $N_{r+s}(\ba;\infty)$ is finite if $r,s\geq 2$.  A natural question to ask is what happens if instead of allowing only $\kappa$ to vary, we also allow $\ba$ to vary.  Namely, we can consider the quantity
\[
\sup_{\ba}N_{r+s}(\ba;\infty)
\]
for fixed $r,s\geq 2$.  We have the following conjecture.

\begin{conjecture}\label{unboundedconjecture}For any positive integers $r\geq s\geq 2$ we have
\[
\sup_{\ba}N_{r+s}(\ba;\infty)=\infty
\]
where the supremum is over all non-negative integer vectors $\ba$.
\end{conjecture}

Although we have not been able to verify this conjecture in full generality, we do have the following support for our conjecture.

\begin{theorem}\label{unboundedtheorems=2}For any integer $r\geq s=2$ we have
\[
\sup_{\ba}N_{r+2}(\ba;\infty)=\infty
\]
where the supremum is over all non-negative integer vectors $\ba$.
\end{theorem}

Theorem \ref{finitenesstheorem} above has the following interesting corollary.

\begin{corollary}Let $\ba=(a_1,\ldots,a_r)$ be as before, and let $r,s\geq 2$.  Then there is a constant $K$ so that for all $\kappa\geq K$, the symplectomorphism class of $(M_{\ba},\omega_{\ba}^\kappa)$ has exactly $N_n(\ba;\infty)$ inequivalent toric structures, for all $\kappa>K$.
\end{corollary}
\begin{proof}By Theorem \ref{finitenesstheorem}, we know that $N_n(\ba;\infty)$ is a finite number.  More specifically, for all $\kappa>(r+1-\tfrac{1}{r})\sigma_1(\ba)-s$, $N_n(\ba;\kappa)=N_n(\ba;\infty)$.  Also, as in Lemma \ref{deformationequivalentlemma}, for each vector $\bb$ so that $(M_{\bb},\omega_{\bb}^\kappa)$ is deformation equivalent to $(M_{\ba},\omega_{\ba}^\kappa)$ for some $\kappa$, there is a constant $C_{\bb}$ so that for all $\kappa>C_{\bb}$, $(M_{\bb},\omega_{\bb}^\kappa)$ is actually symplectomorphic to $(M_{\ba},\omega_{\ba}^\kappa)$.  Furthermore, there is a finite number of such constants $C_{\bb}$, and thus we can define the constant $K$ as
\[
K:=\max\{C_{\bb},(r+1-\tfrac{1}{r})\sigma_1(\ba)-s\}
\]
Then as above, the constant $K$ is as desired.  Namely, for all $\kappa>K$, the symplectomorphism class of $(M_{\ba},\omega_{\ba}^\kappa)$ has exactly $N_n(\ba;\infty)$ toric structures.
\end{proof}

\begin{acknowledgements}This paper would not have been possible without the constant support and advice of my advisor, Dusa McDuff, who first suggested the topic and helped revise several versions of this paper.
\end{acknowledgements}

\section{Definitions and Basic Results}
This section discusses some basic background results that are important for understanding and proving our results above.  It is divided into an introduction to the geometric ideas and the homological ideas we will need.  We now introduce the geometric ideas.

First, we will discuss toric manifolds and polytopes.  We will focus mostly on the case where $M$ is a symplectic toric $\bC P^r$ bundle over $\bC P^s$.  Recall that we say $M$ is a symplectic bundle if $M$ is an $\widetilde{M}$ bundle over $\widehat{M}$ so that $\widetilde{M}$ has a symplectic structure $\omega_0$ and the structure group of the bundle is $\Symp(\widetilde{M})$.  In particular, this implies that each fiber $F_x$ over a point $x\in\widehat{M}$ has a symplectic structure $\omega_x$ so that $i^*(\omega_x)=\omega_0$ where $i$ is the inclusion of the standard fiber.

If $H^1(\widehat{M})=0$, as it is in our case where $\widehat{M}=\bC P^s$, we can piece the forms $\omega_x$ together into a closed form $\tau$ on $M$ so that $\tau$ is non-degenerate on the fibers of $M$.  If also $(\widehat{M},\widehat{\omega})$ is symplectic, then there is a symplectic form $\omega$ on $M$, defined by
\[
\omega=\tau+K\pi^*(\widehat{\omega})
\]
where $\pi:M\rightarrow\widehat{M}$ is the projection and $K\in\bR$.  It is well known that $\omega$ is symplectic for sufficiently large $K$.

Now further assume that that we have Hamiltonian torus actions $\widetilde{T}$, $T$, and $\widehat{T}$ on $\widetilde{M}$, $M$, and $\widehat{M}$ respectively, making them each symplectic toric manifolds.  Then we say that $M$ is a symplectic toric bundle if there is a short exact sequence
\[
\widetilde{T}\rightarrow T\rightarrow\widehat{T}
\]
such that $i:(\widetilde{M},\widetilde{T})\rightarrow (M,T)$ and $\pi:(M,T)\rightarrow(\widehat{M},\widehat{T})$ are equivariant.

Now, let $\Delta$ be the moment polytope of a toric structure on some symplectic toric manifold $(M,\omega,T)$.  We can describe $\Delta$ as
\[
\{x\in t^*|\langle x,\eta_i\rangle\leq\kappa_i\text{ for all $i$}\}
\]
where the $\eta_i$ are the outward primitive integer conormals to the facets of $\Delta$ and the $\kappa_i$ are support constants.

\begin{example}\label{projectiveexample}The moment polytope of $\bC P^n$ will be denoted $\Delta_n$, and is a copy of the standard $n$-simplex when we choose $\eta_i=-e_i$ for $1\leq i\leq n$, $\eta_{n+1}=(1,\ldots,1)$ and all $\kappa_i=1$.  Notice that $\Delta_n$ has edges of length $n+1$ and has volume equal to
\[
\Vol(\Delta_n)=\tfrac{1}{n!}(n+1)^n
\]
\end{example}

We recall that any moment polytope is a simple, smooth, rational polytope.  If $\dim(\Delta)=n$, \textbf{simple} means at each vertex exactly $n$ facets meet, \textbf{rational} means that the conormal vectors to these facets are primitive integral vectors, and \textbf{smooth} means that these vectors form an integer basis of $\bZ^n$. We call such polytopes Delzant polytopes.  The well known Delzant theorem from \cite{Del} says

\begin{theorem}(Delzant)For each Delzant polytope $\Delta$, there is a symplectic toric manifold $M_\Delta$ with moment polytope $\Delta$.  Moreover, $(M,\omega,T)$ is equivariantly symplectomorphic to $(M',\omega',T')$ if and only if $\Delta_M$ and $\Delta_{M'}$ are equivalent under the affine group generated by translations and the action of $GL(n;\bZ)$
\end{theorem}

We are most interested in the case where the manifold $M$ is a symplectic toric $\widetilde{M}$ bundle over $\widehat{M}$.  To study this, we will discuss the notion of a bundle of polytopes.

The general definition of a polytope $\Delta$ being a $\widetilde{\Delta}$ bundle over $\widehat{\Delta}$ given as $3.10$ of \cite{DT} is more complicated than we will need, so we instead summarize some key points.  In particular, we only need the notion of a $\Delta_r$ bundle over $\Delta_s$.

The basic idea is to develop a notion of bundles so that by the Delzant theorem above, a manifold $(M,\omega,T)$ is a symplectic toric $(\bC P^r,\omega_r,T_r)$ bundle over $(\bC P^s,\omega_s,T_s)$ if and only if $\Delta$ is a $\Delta_r$ bundle over $\Delta_s$.  At this point, we recall that $\Delta\subset \mathfrak{t}^*$, where $\mathfrak{t}$ is the Lie algebra of $T$, and similarly for $\Delta_r$ and $\Delta_s$.  Since the moment polytopes are then naturally subsets of the dual spaces to the Lie algebras of the torus actions, we should expect a $\Delta_r$ bundle over $\Delta_s$ to naturally be fibered by $\Delta_s$ over $\Delta_r$, instead of the other way around.  This motivates the following definition.

\begin{definition}\label{bundledefinition}We say that a polytope $\Delta$ is a $\Delta_r$ bundle over $\Delta_s$ if, for some choice of $(\ba;\kappa)$ where $\ba=(a_1,\ldots, a_r)$ and $\kappa\in\bR$ with $\kappa>K_{\ba}:=\sigma_1(\ba)-s$, $\Delta$ is affine equivalent to the polytope $\Delta_{\ba}^\kappa$, which is defined by setting
\begin{equation*}
\eta_i = \left\{
\begin{array}{ll}
-e_i & \text{if } 1\leq i\leq r\\
(1,\ldots,1,0,\ldots,0) & \text{if } i=r+1\\
-e_{i-1} & \text{if } r+2\leq i\leq s+1\\
(-a_1,\ldots,-a_r,1,\ldots,1) & \text{if } i=r+s+2
\end{array} \right.
\end{equation*}
with $\kappa_i=1$ for $1\leq i\leq r+s+1$, and $\kappa_{r+s+2}=\kappa$.
\end{definition}

\begin{remark}\label{bundlerescalingremark}The polytope $\Delta_{\ba}^\kappa$ naturally has the structure of a standard copy of $\Delta_r$ with fibers that are all rescaled copies of $\Delta_s$.  The vector $\ba$ then has a natural interpretation as the slope of the increase of the rescaling as we move in the standard directions in $\Delta_r$, while the number $\kappa$ determines the rescaling over the origin.  We will now take a few moments to show this more explicitly by computation.

We obtain relations on the coordinates $x_i$ of an arbitrary point of $\Delta_{\ba}^\kappa$ by computing $\langle x,\eta_i\rangle$ for each $i$.  We get the inequalities
\begin{equation*}
\begin{array}{ll}
x_i\geq-1,\forall i\\
x_1+\ldots+x_r\leq 1\\
x_{r+1}+\ldots +x_{r+s}\leq \kappa+a_1x_1+\ldots+a_rx_r
\end{array}
\tag{*}\end{equation*}
The first two lines of $(*)$ imply the first $r$ coordinates of $x$, $(x_1,\ldots,x_r)$, form a standard copy of $\Delta_r$, as described in Example \ref{projectiveexample}.  Also, the first and third lines of $(*)$ show that the last $s$ coordinates of $x$, $(x_{r+1},\ldots,x_{r+s})$, form a rescaled copy of $\Delta_s$.  Namely, they form a polytope $\Delta_s^{\kappa,x}$ described as a subset of $\bR^s$ by the conormals
\begin{equation*}
\eta_i=\left\{
\begin{array}{ll}
\eta_i=-e_i,\forall 1\leq i\leq s\\
\eta_{s+1}=(1,\ldots,1)
\end{array}\right.
\end{equation*}
with support constants $\kappa_i=1$ for $1\leq i\leq s$ and $\kappa_{s+1}=\kappa+a_1x_1+\ldots+a_rx_r$.  Thus, $\Delta_s^{\kappa,x}$ is simply a standard simplex with edge length $s+\kappa+a_1x_1+\ldots+a_rx_r$.

Note also that the inequalities $(*)$ justify the restriction that $\kappa>\sigma_1(\ba)-s$.  Indeed, if we assume that $(x_1,\ldots,x_r)=(-1,\ldots,-1)$, then the third inequality of $(*)$ says that
\[
x_{r+1}+\ldots+x_{r+s}\leq\kappa-a_1-\ldots-a_r
\]
But on the other hand, the first line of $(*)$ implies that $x_i\geq -1$, so that
\[
x_{r+1}+\ldots+x_{r+s}\geq s
\]
Thus, to avoid contradiction, we must assume that $\kappa>\sigma_1(\ba)-s$.

Also, note that in our case, we assumed all $a_i\geq 0$, so that in the inequality
\[
x_{r+1}+\ldots+x_{r+s}\leq\kappa+a_1x_1+\ldots+a_rx_r,
\]
the size of the right-hand side increases as the $x_i$ increase.  Thus, the $\Delta_s$ fiber of the point $(-1,\ldots,-1)$ is the smallest fiber.
\end{remark}
We now have the following lemma, which gives the relation between $\bC P^r$ bundles over $\bC P^s$ and $\Delta_r$ bundles over $\Delta_s$ and discusses the effect of increasing $\kappa$ on the symplectic form $\omega_{\ba}^\kappa$

\begin{lemma}\label{bundlecorrespondencetheorem} Let $(M,\omega,T)$ be a symplectic toric manifold with polytope $\Delta$.  Then $(M,\omega,T)$ is a symplectic toric $(\bC P^r,\omega_r,T_r)$ bundle over $(\bC P^s,\omega_s,T_s)$ if and only if $\Delta$ is a $\Delta_r$ bundle over $\Delta_s$ equivalent to $\Delta_{\ba}^\kappa$ for some $(\ba;\kappa)$.  Moreover, for a fixed pair $(\ba;\kappa)$
\[
\omega_{\ba}^{\kappa+K}=\omega_{\ba}^\kappa+\tfrac{K}{s+1}\pi^*(\omega_s).
\]

\end{lemma}
\begin{proof}The first statement is discussed in detail in Remark $5.2$ of \cite{DT}, but is difficult to prove in much generality without the full definition of a bundle of polytopes, which we have omitted.  The idea is to use the full definition of a bundle of polytopes to compute $M$ as a complex manifold.  In fact, in remark $5.2$ of \cite{DT}, it is concluded that if $\ba=(a_1,\ldots,a_r)$,
\[
M_{\ba}=\bC P^r\times_{\bC^*}(\bC^{s+1}\setminus\{0\})
\]
for the following $\bC^*$ actions.  Let $(z_1,\ldots,z_r)$ be coordinates on $\bC P^r$.  Then if $te^{i\theta}$ represents the standard polar form of a number in $\bC^*$, the action on $\bC P^r$ is described by
\[
te^{i\theta}\cdot(z_1\ldots,z_r)=\Bigl((te^{i\theta+a_1})z_1,\ldots,(te^{i\theta+a_r})z_r\Bigr)
\]
On $\bC^{s+1}\setminus\{0\}$, the $\bC^*$ action is described by
\[
te^{i\theta}\cdot(z_1,\ldots,z_{s+1})=\Bigl((te^{i\theta})z_1,\ldots,(te^{i\theta})z_{s+1}\Bigr)
\]
which is the standard $\bC^*$ action.

Furthermore, we have
\[
(M_{\ba},\omega_{\ba}^\kappa)=\Bigl(\bC P^r\times (\bC^{s+1}\setminus\{0\}),\Omega_{\lambda(\kappa)}=\omega_r\times\lambda(\kappa)\omega_0\Bigr)
\]
where $\omega_r$ is the standard form on $\bC P^r$, with the rescaling so that $\omega_r(\ell)=r+1$ with $\ell$ the homology class of a line, $\omega_0$ is the standard form on $\bC^{s+1}$, and $\lambda(\kappa)$ is a rescaling factor determined by $\kappa$.

We seek to compute $\omega_{\ba}^{\kappa+K}$.  As above, $M_{\ba}$ is determined as a complex manifold by the relation
\[
M=\bC P^r\times_{\bC^*}(\bC^{s+1}\setminus\{0\}).
\]
Furthermore, $\omega_{\ba}^{\kappa+K}$ is the reduction of $\Omega_{\lambda(\kappa+K)}$ by the $\bC^*$ action.  Then, an easy computation shows that
\[
\omega_{\ba}^{\kappa+K}-\omega_{\ba}^\kappa=\tfrac{K}{s+1}\omega_s
\]
where $\omega_s$ is the standard form on $\bC P^s$ normalized so that $\omega_s(\ell)=s+1$, as before.  Reordering the terms, we get the desired result.
\end{proof}

We now give a helpful condition for detecting when a polytope $\Delta$ is a $\Delta_r$ bundle over $\Delta_s$.  First, there is the notion of two polytopes being combinatorially equivalent.

\begin{definition}\label{combinatoriallyequivalentdefinition}Two polytopes $\Delta$ and $\Delta'$ are said to be \emph{combinatorially equivalent} if there is a bijection $\phi$ between the facets $F_i$ of $\Delta$ and $F_i'$ of $\Delta'$ with $\phi(F_i)=F_i'$ such that for each index set $I$
\[
\bigcap_{i\in I}F_i\neq\emptyset\Longleftrightarrow\bigcap_{i\in I}F_i\neq\emptyset
\]
\end{definition}

McDuff and Tolman prove the following lemma in \cite{DT}

\begin{lemma}\label{bundlelemma}(\cite{DT} Lemma $4.10$) Let $\Delta$ be a polytope which is smooth and combinatorially equivalent to $\Delta_r\times\Delta_s$.  Then $\Delta$ is a $\Delta_r$ bundle over $\Delta_s$ or a $\Delta_s$ bundle over $\Delta_r$.
\end{lemma}

For the rest of the paper, we will only be interested in polytopes $\Delta$ which are $\Delta_r$ bundles over $\Delta_s$ for some choice of $r,s$, which as in Definition \ref{bundledefinition} are determined by pairs $(\ba;\kappa)$.

Using the above presentation we see that the vector $a=(a_1,\ldots,a_r)$ determines the underlying bundle structure of the corresponding manifold $M$, while the constant $\kappa$ determines how much of the structure of the base $\widehat{M}$ is pulled back to the total space.

We reinterpret $\kappa$ in terms of the volume of the polytope to relate the above choice of $(\ba;\kappa)$ to the choice given at the beginning of the paper.  We have the following lemma.

\begin{lemma}\label{volumetheorem}
\[
Vol(\Delta_{\ba}^\kappa)=\tfrac{1}{r!}\tfrac{1}{s!}(r+1)^r(\kappa+s)^s
\]
\end{lemma}
\begin{proof}Consider the polytope $\Delta_{\ba}^\kappa$.  As we saw before, this geometrically looks like a standard copy of $\Delta_r$ with a rescaled copy of $\Delta_s$ over each point.  The point $(0,\ldots,0)$ is the barycenter of the standard copy of $\Delta_r$, and the copy of $\Delta_s$ over this point is the rescaled polytope $\Delta_s^\kappa$ discussed in Remark \ref{bundlerescalingremark}.  We recall that it has the form of a standard $s$ simplex with side length $s+\kappa$.  Now, since the sizes of the rescaled copies of $\bC P^s$ over the base copy of $\bC P^r$ depend linearly on the coordinates in $\bC P^r$, and $\Delta_s^\kappa$ is the $\Delta_s$ over the barycenter, we know that
\[
Vol(\Delta_{\ba}^\kappa)=Vol(\Delta_r)Vol(\Delta_s^\kappa).
\]
However, a simple geometric argument shows that
\[
Vol(\Delta_s^\kappa)=\tfrac{1}{s!}(\kappa+s)^s\qquad Vol(\Delta_r)=\tfrac{1}{r!}(r+1)^r.
\]
\end{proof}
Thus, the tuple $(\ba;\kappa)$ can be interpreted as $\ba$ determining the bundle structure and $\kappa$ determining the volume.

Also, as we see in section $2.4$ of \cite{D}, we can restrict to the case where $a_i\geq 0$ for all $i$.  To see this, recall from before that $\Delta_{\ba}^\kappa$ has a standard copy of $\Delta_r$ with each point having a rescaled $\Delta_s$ over it.  Also, we know that for any vertex of $\Delta$, we can choose cordinates around that vertex so that the edge directions from that vertex are the standard vectors $e_1,\ldots,e_n$.  If we choose coordinates for $\Delta$ around the point of $\Delta_r$ with the "smallest" copy of $\Delta_s$, then by the interpretation of $-a_i$ as the slopes of linear changes in the standard coordinate directions, we have $-a_i\leq 0$ for all $i$, which means $a_i\geq 0$ for all $i$.

We now prove Lemma \ref{toricstructurelemma}, which we recall states that
\[
(M_{\ba},\omega_{\ba}^\kappa,T_{\ba})\cong(M_{\bb},\omega_{\bb}^\kappa,T_{\bb})~\Longleftrightarrow~(\ba;\kappa)=(\bb;\kappa),
\]
where $\cong$ denotes the relation of equivariant symplectomorphism.

\begin{proof}[Proof of Lemma \ref{toricstructurelemma}]First, we notice that if $(\ba;\kappa)=(\bb;\kappa')$, then the manifolds are equivalent.  It remains to show that if the manifolds are equivalent, then $(\ba;\kappa)=(\bb;\kappa')$.  In particular, we show that if $\Delta_{\ba}^\kappa$ is affine equivalent to $\Delta_{\bb}^\kappa$, then $(\ba;\kappa)=(\bb;\kappa')$.

Since $\Delta_{\ba}^\kappa$ is affine equivalent to $\Delta_{\ba}^\kappa$ and affine equivalences preserve volume, $\Vol(\Delta_{\ba}^\kappa)\neq \Vol(\Delta_{\bb}^{\kappa'})$.  Then, a simple application of Lemma \ref{volumetheorem} shows that $\kappa=\kappa'$.  We now show that $\ba=\bb$.

As in Remark \ref{bundlerescalingremark}, the polytope $\Delta_{\ba}^\kappa$ consists of a standard copy of $\Delta_r$ with a rescaled copy of $\Delta_s$ over each point.  Furthermore, as we move in the direction $e_i$ in the base copy of $\Delta_r$, the edge lengths of the specific copy of $\Delta_s$ increase linearly with slope $a_i$.  Thus, we have exactly $r+1$ different $s$-dimensional faces of $\Delta$, which are all copies of $\Delta_s$ of various sizes sitting over the $r+1$ vertices of this $\Delta_r$.  In particular, combining remark \ref{bundlerescalingremark} with lemma \ref{volumetheorem}, we can easily compute that volume of the smallest such $\Delta_s$ is $\tfrac{1}{s!}(\kappa+s-\sigma_1(\ba))$, while the volumes of the other $s$ faces will be $\tfrac{1}{s!}(\kappa+s-\sigma_1(\ba)+a_i)$.

Similarly, in $\Delta_{\bb}^\kappa$, there are $r+1$ different $s$-dimensional faces which are copies of $\Delta_s$, and their volumes are given by $\tfrac{1}{s!}(\kappa+s-\sigma_1(\bb))$ and $(\tfrac{1}{s!}(\kappa+s-\sigma_1(\bb)+b_i)$.  Now, if there is an affine equivalence from $\Delta_{\ba}^\kappa$ to $\Delta_{\bb}^\kappa$, it would have to send the $r+1$ copies of $\Delta_s$ in $\Delta_{\ba}^\kappa$ to the corresponding copies of $\Delta_s$ in $\Delta_{\ba}^\kappa$ while preserving their volumes.  In particular, by the above computations, this implies that $\sigma_1(\ba)=\sigma(\bb)$ and furthermore that for each $i$, there is a $j$ so that $a_i=b_j$.  But the assumption that $0\leq a_1\leq\ldots\leq a_r$ and $0\leq b_1\leq\ldots\leq b_r$ implies that for each $i$, we can make the choice so that $a_i=b_i$.  Thus, if the polytopes are affine equivalent, then $(\ba;\kappa)=(\bb;\kappa')$, as desired.
\end{proof}

Now we will get into some of the more technical lemmas we will need for the proofs of our results.

\begin{lemma}\label{B=0homologicallemma}Let $r,s\geq 1$ be integers with $r>1$, and $\ba=(a_1,\ldots,a_r)$ be a non-negative integer vector with some $a_i\neq 0$.  Assume $H^*(M;\bZ)$ is isomorphic to the graded ring generated by $\alpha_0$ and $\beta_0$ of $H^2(M)$ with relations
\[
\alpha_0^{s+1}=0,\qquad \beta_0\prod_{i=1}^r(\beta_0-a_i\alpha_0)=0,
\]
Then if there exist integers $A,B$ so that $(A\alpha_0+B\beta_0)^{s+1}=0$, we must have $B=0$.
\end{lemma}
\begin{proof}
This a slight restatement of Lemma $6.2$ in \cite{CMS}.  We follow their proof closely.  Since $(A\alpha_0+B\beta_0)^{s+1}=0$, $(A\alpha_0+B\beta_0)^{s+1}$ must be a consequence of our other relations.  Namely, there exists $C,D$ so that
\[
(A\alpha_0+B\beta_0)^{s+1}-C\alpha_0^{s+1}=D\beta_0\prod_{i=1}^r(\beta_0-a_i\alpha_0),
\]
where $C$ is an integer and $D$ is an integer polynomial in $\alpha_0$ and $\beta_0$ of degree $s-r$ if $r\leq s$, and $D=0$ if $r>s$.

If $r>s$, we then have $(A\alpha_0+B\beta_0)^{s+1}-C\alpha_0^{s+1}=0$, which gives $B=0$ and $C=A^{s+1}$, as desired.

Consider now $r\leq s$.  Suppose first that $A=0$.  Since the right hand side has no pure $\alpha_0$ terms and $A=0$, we must have $C=0$ and the left hand side is only a $\beta_0^{s+1}$ term.  But some $a_i\neq 0$, so that the right hand side has a non-zero $\beta_0^{r+1}$ term and a non-zero $\alpha_0\beta_0^r$ term, which is a contradiction. Thus, $A\neq 0$.  Now, since the right hand side has no pure $\alpha_0$ terms and $A\neq 0$, we must have $C=A^{s+1}$ to cancel the $\alpha_0^{s+1}$ term from the left hand side.  If now $B\neq 0$, the remaining terms on the left hand side can be expressed as a polynomial in $\alpha_0$ and $\beta_0$ with no more than $2$ linear factors when optimally factored, while the right hand side has at least three linear factors since some $r>1$, so that the two polynomials can never be equal.  We briefly describe the
factorization of the LHS.  First, Let $A\alpha_0=X$ and $B\beta_0=Y$.  Then, since $C=A^{s+1}$, the LHS can be expressed as
\[
(X-Y)^{s+1}-X^{s+1}
\]
and this has no more than $2$ linear factors, as claimed.  This contradiction establishes that $B=0$.
\end{proof}

We will now prove Proposition \ref{mainlemma}, which we use heavily in the proofs of our main theorems.

\begin{proof}[Proof of Proposition \ref{mainlemma}]We will prove that $(1)\Rightarrow(2)\Rightarrow(4)\Rightarrow(1)$ and also $(2)\Leftrightarrow(3)$.

First, we prove that $(1)\Rightarrow(2)$.  This is the hardest direction of the proof, and we will break it into three cases.  First assume that $r>1$.  This proof is taken from Theorem $6.1$ of \cite{CMS}.  The Stanley-Reisner presentation of $H^*(M_{\ba};\bZ)$ on $\Delta_{\ba}$ gives generators $\alpha_0$ and $\beta_0$ for $H^*(M_{\ba};\bZ)$ satisfying
\begin{equation*}
\alpha_0^{s+1}=0
\tag{$\alpha_0$}\end{equation*}
\begin{equation*}
\beta_0\prod_{i=1}^r(\beta_0-a_i\alpha_0)=0
\tag{$\beta_0$}\end{equation*}
and similarly, from the polytope $\Delta_{\bb}$, we get generators $\alpha$ and $\beta$ of $H^*(M_{\bb};\bZ)$ with the relations
\begin{equation*}
\alpha^{s+1}=0
\tag{$\alpha$}\end{equation*}
\begin{equation*}
\beta\prod_{i=1}^r(\beta-b_i\alpha)=0.
\tag{$\beta$}\end{equation*}
Since $H^*(M_{\ba})$ is isomorphic to $H^*(M_{\bb})$, there exist integers $A,B,C,D$ with $AD-BC=1$ so that
\[
\alpha_0=A\alpha+B\beta,\qquad\beta_0=C\alpha+D\beta
\]
Using $\alpha_0^{s+1}=0$ and Lemma \ref{B=0homologicallemma}, we conclude that $B=0$, so that $A=D=\pm 1$.  Moreover we can arrange $A=D=1$ by possibly changing the signs of both $\alpha$ and $\beta$.  Now we substitute $\beta_0=C\alpha+\beta$ and $\alpha_0=\alpha$ into the relation $(\beta_0)$, and since the relation $(\beta_0)$ must equal the relation $(\beta)$, we know that the two polynomials are equal as polynomials in $\beta$.  Substituting the specific value $\beta=1$, we obtain the relation
\begin{equation*}
\prod_{i=0}^r(1+(-a_i+C)\alpha)=\prod_{i=0}^r(1-b_i\alpha)
\tag{*}\end{equation*}
where we assume that $a_0=b_0=0$.

But the left hand side is just the total Chern class of the bundle
\[
[L_0\oplus L_{-a_1}\oplus\ldots\oplus L_{-a_r}]\otimes L_C
\]
while the right hand side is the total Chern class of the bundle
\[
L_0\oplus L_{-b_1}\oplus\ldots\oplus L_{-b_r}.
\]
Thus, since these two bundles have the same total Chern class and are sums of line bundles, they are isomorphic as vector bundles, i.e.
\[
[L_0\oplus L_{-a_1}\oplus\ldots\oplus L_{-a_r}]\otimes L_C\cong(L_0\oplus L_{-b_1}\oplus\ldots\oplus L_{-b_r}).
\]
But the above shows that
\[
\bP(L_0\oplus L_{-a_1}\oplus\ldots\oplus L_{-a_r})=\bP(L_0\oplus L_{-a_1}\oplus\ldots\oplus L_{-a_r}),
\]
as desired.

Now, consider the case $r=s=1$.  Then $\ba=(a)$, $\bb=(b)$, and we know that $M_{\ba}$ and $M_{\bb}$ are just the Hirzebruch surfaces $H_a$ and $H_b$ respectively.  Repeated application of Lemma \ref{symplectomorphismlemma} then implies that $H_a$ is symplectomorphic to $H_b$ if $b-a$ is even.  A simple computation shows that $H_0\ncong H_1$, so that in fact $H_a$ is symplectomorphic to $H_b$ if and only if $b-a$ is even.  In particular, $H^*(M_{\ba};\bZ)\cong H^*(M_{\bb};\bZ)$ if and only if $b-a$ is even.  But then $C=\tfrac{a-b}{2}$ is an integer.  Now let $a_0=b_0=0$ $\alpha$ is as before.  In particular, $\alpha^2=0$ and a simple computation shows
\[
\prod_{i=0}^1(1+(-a_i+C)\alpha)=(1+C\alpha)(1+(C-a)\alpha)=1+(2C-a)\alpha=1-b\alpha=\prod_{i=0}^1(1-b_i\alpha),
\]
which implies condition $(2)$ as above.

Lastly, consider the case $r=1$, $s\geq 2$.  As before, $\ba=(a)$ and $\bb=(b)$.  Using the Stanley-Reisner presentation, we get $\alpha_0$, $\beta_0$, $\alpha$ and $\beta$ as before, with integers $A$, $B$, $C$, and $D$ with $AD-BC=1$ and
\[
\alpha_0=A\alpha+B\beta,\qquad\beta_0=C\alpha+D\beta
\]
Now, recall from equations $(\beta_0)$ and $(\beta)$ above that $\beta_0(\beta_0-b\alpha_0)=0$ and $\beta(\beta-a\alpha)=0$Substituting from the above, expanding, and simplifying, we get
\begin{align*}
0&=(C\alpha+D\beta)((C\alpha+D\beta)-b(A\alpha+B\beta))\\
&=(C(C-bA))\alpha^2+(C(D-bB)+D(C-bA))\alpha\beta+(D(D-bB))\beta^2\\
&=(C(C-b)A)\alpha^2+(C(D-bB)+D(C-bA)+a(D(D-bB)))\alpha\beta.
\end{align*}
Also, since $s\geq 2$, equation $(\alpha)$ tell us that $\alpha^2\neq 0$ and $\alpha\beta\neq 0$, which means that
\[
C(C-bA)=0\qquad C(D-bB)+D(C-bA)=-a(D(D-bB)).
\]
$C(C-bA)=0$ implies that either $C=0$ or $C=bA$.  If $C=0$, then by $AD-BC=1$, we know that $A=D=\pm 1$, where by changing signs of $\alpha$ and $\beta$ if necessary, we can arrange $A=D=1$.  Substituting into the above, this tells us that
\[
-b=-a(1-bB)
\]
so that $b$ divides $a$.

Now, assume $C=bA$.  Then $AD-BC=1$ implies that $AD-bAB=1$ so $A(D-bB)=1$ which says that $A=D-bB=\pm 1$, where as before we can arrange $A=D-bB=1$.  Then substituting as before, we get
\[
b=-aD
\]
so that again $b$ divides $a$.  Thus, in either case, we have $b$ divides $a$.  But switching the roles of $a$ and $b$, we clearly also have $a$ divides $b$, so that in fact, $a=b$.  Thus, we clearly have
\[
\prod_{i=0}^1(1-a_i\alpha)=\prod_{i=0}^r(1-b_i\alpha)
\]
which implies condition $(2)$ as before.  Thus, we conclude that $(1)\Rightarrow (2)$.

We next prove $(2)\Rightarrow(4)$.  Now, since we are assuming that $\bP(L_0\oplus L_{-a_1}\oplus\ldots\oplus L_{-a_r})$ is isomorphic to $\bP(L_0\oplus L_{-b_1}\oplus\ldots\oplus L_{-b_r})$ as a projective vector bundle, we know that there is some $C$ so that $(L_0\oplus L_{-a_1}\oplus\ldots\oplus L_{-a_r})\otimes L_C$ is isomorphic to $L_0\oplus L_{-b_1}\oplus\ldots\oplus L_{-b_r}$ as vector bundles, which implies that they have the same total Chern class, which gives us the relation

\begin{equation*}
\prod_{i=0}^r(1+(-a_i+C)\alpha)=\prod_{i=0}^r(1-b_i\alpha)
\tag{*}\end{equation*}
where we assume that $a_0=b_0=0$ as before.  Here, we have assumed that $\alpha$ is the standard generator of $\bC P^s$, so that $L_{\ba}$ is the line bundle over $\bC P^s$ with first Chern class given by $a\alpha$.  Since $\alpha^{s+1}=0$, we know by expanding and comparing coefficients of $\alpha^i$ that the above equation is true if and only if:
\[
\sigma_i(C,C-a_1,\ldots,C-a_r)=\sigma_i(0,-b_1,\ldots,-b_r)\quad 1\leq i\leq \min\{r+1,s\}
\]
which in turn is true if and only if
\[
\sigma_i(C,a_1+C,\ldots,a_r+C)=\sigma_i(0,b_1,\ldots,b_r)
\]
where we have replaced $-C$ by $C$ as the arbitrary constant.  This finishes the proof that $(2)\Rightarrow(4)$.

Next we show that $(4)\Rightarrow(1)$.  By $(4)$, we know that there exists a constant $C$ so that
\begin{equation*}
\prod_{i=0}^r(1+(-a_i+C)\alpha)=\prod_{i=0}^r(1-b_i\alpha)
\tag{*}\end{equation*}
where as above, $\alpha$ is the standard generator of $H^2(\bC P^s)$.  As before, this implies condition $(2)$, which implies condition $(1)$.

It remains to show $(2)\Leftrightarrow(3)$.  In both cases, the manifold $M$ is a smooth $\bC P^r$ bundle over $\bC P^s$.  The difference is that in $(2)$, we are considering it as a projective vector bundle, so that the structure group of the bundle is $PU(r+1)$, whereas in condition $(3)$, we are considering it as a symplectic bundle, so that the structure group of the bundle is $\Symp(\bC P^r)$.  Thus, the fact that $(2)\Rightarrow(3)$ follows immediately from the fact $PU(r+1)\subset\Symp(\bC P^r)$.

It remains to show that $(3)\Rightarrow(2)$.  However, as is shown in \cite{Rez}, there is a natural extension of the notion of Chern classes to symplectic bundles.  Thus, since we have two isomorphic symplectic bundles, they have equal total Chern classes in the symplectic sense, which implies that they have equal total Chern class in the projective sense.  Thus, there is a constant $C$ so that the bundles $(L_0\oplus L_{-a_1}\oplus\ldots\oplus L_{-a_r})\otimes L_C$ and $(L_0\oplus L_{-b_1}\oplus\ldots\oplus L_{-b_r})$ have the same total Chern class, which as before implies that they are isomorphic as vector bundles.  This in turn implies the condition $(2)$.
\end{proof}

Lastly, we need a couple more lemmas to characterize the possible moment polytopes of toric structures on symplectic toric bundles.  First, we recall the following theorem from \cite{Tim}

\begin{lemma}\label{combinatoriallyequivalentlemma}(\cite{Tim} Prop $1.1.1$) Let $\Delta$ be a polytope of dimension $n$ with $n+2$ facets.  Then there exists and $k,m$ with $k+m=n$ so that $\Delta$ is combinatorially equivalent to $\Delta_k\times\Delta_m$.
\end{lemma}

We use this to prove the following fundamental lemma.

\begin{lemma}\label{toricclassifyinglemma}
If $(M^{2n},\omega,T)$ is a symplectic toric manifold with $\dim H^2(M)=2$, then $M$ is a $\bC P^r$ bundle over $\bC P^s$, and hence symplectomorphic to some $(M_{\ba},\omega_{\ba}^\kappa)$.  Moreover, if $\ba\neq 0$, any other toric structure on $M$ is a $\bC P^r$ bundle over $\bC P^s$ for the same $r,s$.
\end{lemma}
\begin{proof}This proof follows the proof of Corollary $6.3$ in \cite{CMS}.
By assumption, $\dim H^2(M)=2$, and therefore $\Delta_M$ has $\dim\Delta_M +\rank(H^2(M))=n+2$ facets, which by Lemma \ref{combinatoriallyequivalentlemma} tells us it is combinatorially equivalent to some $\Delta_r\times\Delta_s$ with $r+s=n$.  Since it is also smooth, Lemma \ref{bundlelemma} says $\Delta_M$ is a $\Delta_r$ bundle over $\Delta_s$ for some choice of $r$ and $s$ with $r+s=n$, which implies that $M$ is a $\bC P^r$ bundle over $\bC P^s$ by Lemma \ref{bundlecorrespondencetheorem}.  As in Definition \ref{bundledefinition}, this bundle is determined by a pair $(\ba;\kappa)$ where $\ba=(a_1,\ldots,1_r)$ can be chosen so that $a_i\geq 0$.

Now, assume that some $a_i\neq 0$ and that we have some other toric structure generating a polytope $\Delta'$.  By the above, $\Delta'$ is a $\Delta_k$ bundle over $\Delta_m$ where $k+m=n$, and hence is determined by a pair $(b;\kappa)$. We show that $k=r$ and $m=s$.  Comparing information about Betti numbers, we can easily conclude that $r+s=m+k$ and $(1+r)(1+s)=(1+k)(1+m)$ so that $\{r,s\}=\{k,m\}$.  We show that we can arrange $m=s$.

To see this, assume that $m=r$, so that $k=s$.  If $r=s$, there is nothing to prove.  First, assume $r<s$.  Now, since $M$ is a $\bC P^k$ bundle over $\bC P^m$, there is an element $\gamma$ in $H^2(M;\bZ)$ so that $\gamma^{m+1}=0, \gamma\neq 0$.  But $\gamma^{m+1}=\gamma^{r+1}$, and $r<s$, therefore $\gamma^s=0$.  But $M$ is a $\bC P^r$ bundle over $\bC P^s$ determined by the vector $\ba$, so as in the proof of Proposition \ref{mainlemma}, we know that $H^*(M;\bZ)\cong H^*(M_{\ba};\bZ)$ has generators $\alpha_0$ and $\beta_0$ with relations
\[
\alpha_0^{s+1}=0,\qquad \prod_{i=0}^r(\beta_0-a_i\alpha_0)=0.
\]
We recall now that we have assumed that some $a_i$ is not zero.  We claim that an element $\gamma$ as above cannot exist.  Indeed, if it did, then we would have $\gamma=A\alpha_0+B\beta_0$ with $\gamma^{s}=0$, and in particular, $\gamma^{s+1}=0$.  Therefore, by Lemma \ref{B=0homologicallemma} above, we have $B=0$, so that $\gamma=A\alpha_0$ and $\gamma^s=0$, with $\gamma\neq 0$.  Since $\gamma\neq 0$, $A\neq 0$ and therefore $\gamma^s=0$ implies that $\alpha_0^s=0$, which is a contradiction.  Therefore, $m=s$ and $k=r$, as required.

Now, consider the case where $r>s$.  Since $k=s$ and $m=r$, we then have $k<m$.  There are two cases to consider.  First, assume $\bb\neq 0$.  Thus, some $b_i$ is non-zero, and we can run the above argument with the roles of $\ba$, $r$, and $s$ replaced by $\bb$, $k$, and $m$ to get the desired result.

Now, if $\bb=0$, then our $\Delta_k$ bundle over $\Delta_m$ is actually $\Delta_k\times\Delta_m=\Delta_m\times\Delta_k$ which is also a $\Delta_m$ bundle over $\Delta_k$, hence a $\Delta_r$ bundle over $\Delta_s$, as desired.
\end{proof}

Using this, we can now prove Theorem \ref{toricclassifyingtheorem}, which we recall said that any toric symplectic manifold $(M,\omega,T)$ with $\dim H^2(M)=2$ is equivariantly symplectomorphic to the bundle $(M_{\ba},\omega^\kappa_{\ba},T_{\ba})$ for a unique tuple $(\ba;\kappa)$ with $0\leq a_1\leq\ldots\leq a_r$.

\begin{proof}[Proof of Theorem \ref{toricclassifyingtheorem}]
Since $\dim H^2(M)=2$, Lemma \ref{toricclassifyinglemma} then implies that $\Delta_M$ is a $\Delta_r$ bundle over $\Delta_s$ determined by some tuple $(\ba;\kappa)$ with $0\leq a_1\leq a_r$, and in fact that $(M,\omega,T)$ is equivariantly symplectomorphic to $(M_{\ba},\omega_{\ba}^\kappa,T_{\ba})$.  Lemma \ref{toricstructurelemma} implies that the tuple $(\ba;\kappa)$ determined in this fashion is in fact uniquely determined.
\end{proof}

\section{Equivalence Relations on Toric Symplectic Manifolds}

We now prove Theorem \ref{symplectomorphismtheorem}, which we recall said that if $(M_{\ba},\omega_{\ba}^\kappa)$ and $(M_{\bb},\omega_{\bb}^\kappa)$ are $\bC P^r$ bundles over $\bC P^s$ with $s=1$ determined by vectors $(\ba;\kappa)$ and $(\bb;\kappa)$, then they are isomorphic as symplectic bundles if and only if they are actually symplectomorphic.  First, we will prove a special case of this, which will act as a technical lemma.

\begin{lemma}\label{symplectomorphismlemma}Let $\ba=(a_1,\ldots,a_r)$ and let $(\ba;\kappa)$ determine the symplectic bundle $(M_{\ba},\omega_{\ba}^\kappa)$, as before, where we assume that $M_{\ba}$ is a $\bC P^r$ bundle over $\bC P^1$.  Now, assume that either $\ba'=(a_1+1,\ldots,a_i+2,\ldots,a_r+1)$ for some $i$ or that $\ba'=(a_1,\ldots,a_i+1,\ldots,a_j-1,\ldots,a_r)$ for some $i,j$.  Then $(M_{\ba},\omega_{\ba}^\kappa)$ and $(M_{\ba'},\omega_{\ba'}^\kappa)$ are symplectomorphic.
\end{lemma}

\begin{proof}
We will prove this theorem in two parts, corresponding to the cases where $\ba'=(a_1+1,\ldots,a_i+2,\ldots,a_r+1)$ for some $i$ or where $\ba'=(a_1+1,\ldots,a_i+1,\ldots,a_j-1,\ldots,a_r+1)$.  Both parts will use the same basic symplectomorphism technique, which we describe below.

Recall as in definition \ref{bundledefinition} that $\Delta$, a $\Delta_r$ bundle over $\Delta_1$, has coordinates $(x_1,\ldots,x_r,z)$ where $(x_1,\ldots,x_r)$ are coordinates on the standard $\Delta_r$ and $z$ will be thought of as the vertical direction, describing the copies of $\Delta_1$ over various points of the base copy of $\Delta_r$.  Recall also there there is a moment map, denoted $\Phi:M_{\Delta}\rightarrow\Delta$ which takes $M_{\Delta}$ to $\Delta$.  Let $\mathcal{H}$ be any hyperplane transverse to the $z$ direction with conormal $\eta_{\mathcal{H}}=(b_1,\ldots,b_r,1)$, with $b_i$ integers.  Consider the intersection of the hyperplane $\mathcal{H}$ with the polytope $\Delta$.  This gives a polytope $\Delta_{\mathcal{H}}$, which is still Delzant because the $b_i$ are integers.

Since $\mathcal{H}$ is transverse to the vertical $z$ direction, $\Delta_{\mathcal{H}}$ effectively splits the polytope $\Delta$ into a top half and a bottom half.  The polytope $\Delta$ is then described by taking the top half and bottom half and gluing them together along $\Delta_{\mathcal{H}}$ by the identity.  Now, consider an affine equivalence of the polytope $\Delta_{\mathcal{H}}$, which we will denote $\phi'$.  We can then define a polytope $\Delta'$ by taking the top half and bottom half, and gluing them together along $\Delta_{\mathcal{H}}$ by the affine equivalence $\phi'$ instead of the identity.  Since the map $\phi'$ is an affine equivalence, $\Delta'$ is evidently still a Delzant polytope.  An example of this is shown in Figure $1$.

We briefly explain why $M_{\Delta}$ and $M_{\Delta'}$ are symplectomorphic.  To do this, we redescribe the above process in a way that is the same symplectically, but not torically.  Namely, we will look on the level of the manifolds, not the polytopes.  First, we consider the hyperplane $Q=\Phi^{-1}(\mathcal{H})$ in $M_{\Delta}$, and thicken it by taking $Q\times\{(0,\epsilon)\}$ and intersecting this with $M_{\Delta}$.  As before, this section of the manifold effectively divides M into a top and bottom half, with the attaching maps to $Q\times\{(0,\epsilon)\}$ at $Q\times\{0\}$ and $Q\times\{\epsilon\}$ being the identity.  We can then symplectically isotop  $Q\times\{(0,\epsilon)\}$ to a thickened hyperplane $Q'\times\{(0,\epsilon)\}$ by an isotopy $\Psi$ where $Q'\times\{\epsilon\}$ is equal to $Q\times\{\epsilon\}$, while $Q'\times\{0\}$ is equivariantly symplectomorphic to $Q\times\{0\}$ by the map $\Phi^*\phi'$, the lift of the affine equivalence $\phi'$.  By doing this we can produce a manifold $M'$ by letting $M'=M$ both above and below $Q\times\{(0,\epsilon)\}$, but replacing $Q\times\{(0,\epsilon)\}$ with $Q'\times\{(0,\epsilon)\}$, with the attaching map to the top half at $Q'\times\{\epsilon\}$ being the identity as before, while the attaching map to the bottom half at $Q'\times\{\epsilon\}$ is the map $\Phi^*\phi'$.  $M_{\Delta}$ and $M'$ are then isotopic, hence symplectomorphic, by the isotopy $\Psi'$ which equals the identity on the top half and bottom half, and which isotops $Q\times\{(0,\epsilon)\}$ to $Q'\times\{(0,\epsilon)\}$ by the isotopy $\Psi$.  However, by construction $M'$ is symplectomorphic to $M_{\Delta'}$, which implies that $M_{\Delta}$ and $M_{\Delta'}$ are symplectomorphic, as desired.

\begin{figure}[t]
\includegraphics{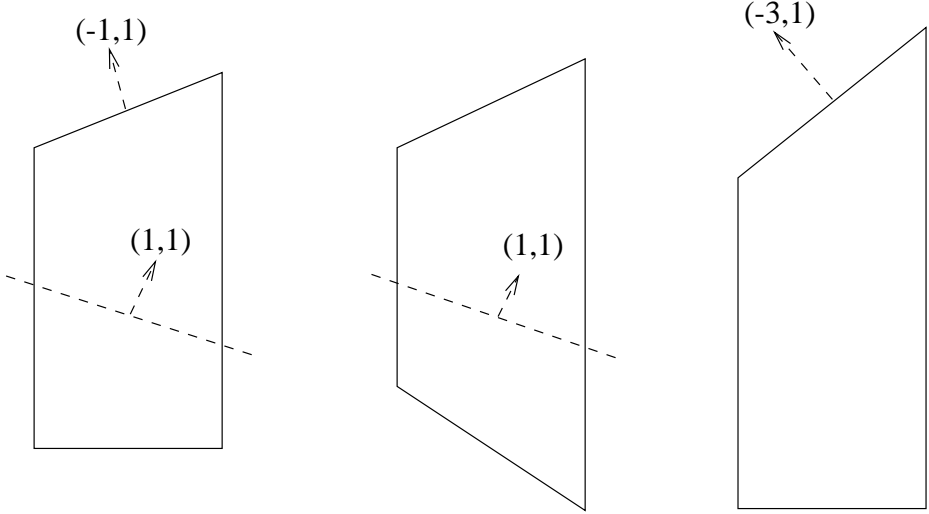}
\caption{Example of Lemma \ref{symplectomorphismlemma} with $r=s=1$, $\ba=1$, and $\eta_{\mathcal{H}}=(1,1)$.  The first figure is $\Delta_{(1)}^\kappa$, the dotted line in the first two figures represents the hyperplane $\mathcal{H}$, the second figure is $\Delta'$, and the third figure is $\Delta_{(3)}^\kappa$.  Notice that $\Delta'$ is affine equivalent to $\Delta_{(3)}^\kappa$}
\end{figure}

To complete the proof, we need only show that we can choose the hyperplane $\mathcal{H}$ and affine equivalence of $\Delta_{\mathcal{H}}$ in such a way that we can obtain $\Delta'=\Delta_{\ba'}^\kappa$ where $\ba'$ is one of the vectors from before.  Before we do this, we first notice that since $\Delta$ is a $\Delta_r$ bundle over $\Delta_1$, if we take $\mathcal{H}$ transverse to the $z$ direction and intersect it with $\Delta$, then $\Delta_{\mathcal{H}}$ is a simply a copy of $\Delta_r$.  We will label the vertices of the standard $\Delta_r$ as $v_0,\ldots,v_r$ where $v_0=(-1,\ldots,-1)$ and $v_i=(-1,\ldots,n,\ldots,-1)$ for $1\leq i\leq r$.

First, we consider vectors of the form $\ba'=(a_1+1,\ldots,a_i+2,\ldots,a_r+1)$.  To show that $\Delta'=\Delta_{\ba'}^\kappa$, we will consider the hyperplane with conormal vector $(1,\ldots,1)$.  Recall from Remark \ref{bundlerescalingremark} that a $\Delta_r$ bundle over $\Delta_1$ can be thought of as a copy of $\Delta_r$ fibered by vertical copies of $\Delta_1$, where the value of $a_i$ is the slope of increase of the sizes of $\Delta_1$ along the edge from $v_0$ to $v_i$.  Thus, to compute the value of $a_i$, we only need to know the size of the vertical edge over each vertex of $\Delta_r$.  It can then be easily computed that if we take the hyperplane described by $(1,\ldots,1)$ as above and take an affine equivalence of $\Delta_{\mathcal{H}}$ which takes the vertex of $\Delta_{\mathcal{H}}$ over $v_0$ to the vertex of $\Delta_{\mathcal{H}}$ over $v_i$, then this shortens the vertical edge over $v_0$ by $1$ unit, lengthens the vertical edge over $v_i$ by $1$, and fixes all other lengths.  This corresponds exactly to changing $\ba$ to $\ba'$.

Consider now the vectors of the form $\ba'=(a_1,\ldots,a_i-1,\ldots,a_j+1,\ldots,a_r)$ and take the hyperplane with conormal vector $(0,\ldots,-1,\ldots,0,1)$ where the $-1$ is in the $i^{th}$ slot.  Then as above, it can be easily computed that by taking an affine equivalence of $\Delta_{\mathcal{H}}$ which takes the vertex of $\Delta_H$ above $v_i$ to the vertex of $\Delta_H$ above $v_j$, we shorten the vertical edge over $v_i$ by $1$ unit, lengthen the vertical edge over $v_j$ by $1$ unit, and fix all other lengths.  Again, this corresponds exactly to changing $\ba$ to $\ba'$, which completes the proof.
\end{proof}

\begin{remark}It can be shown that the above argument only works in the $s=1$ case.  Indeed, if we try to run the above argument in the $s>1$ case, we will find that $\Delta_{\mathcal{H}}$ will correspond to a certain $\Delta_r$ bundle over $\Delta_{s-1}$ where $s-1>0$.  In the $s=1$ case, we had $\Delta_{\mathcal{H}}$ as a $\Delta_r$ bundle over $\Delta_0$, which is just a copy of $\Delta_r$, which has plenty of affine symmetries.  In fact, in $\Delta_r$, there is an affine symmetry which swaps any two vertices.  However, $\Delta_r$ bundles over $\Delta_{s-1}$ with $s-1>0$ have very few affine symmetries.  The only time when $\Delta_{\ba}^\kappa$ will have a symmetry is when some $a_i=0$ or when some $a_i=a_j$.  However, in our case it is easy to check that if we arrange our hyperplane $\mathcal{H}$ to have $\Delta_{\mathcal{H}}$ have one of these symmetries, then in fact the polytopes $\Delta$ and $\Delta'$ from before are affine equivalent.  More specifically, the affine equivalence $\phi$ of $\Delta_{\mathcal{H}}$ could be extended to a global affine equivalence of either the top half or bottom half, which obviously would imply that $\Delta$ and $\Delta'$ are affine equivalent.  In other words, if $s>1$, this symplectomorphism technique only picks up the equivariant symplectomorphisms corresponding to coordinate changes on the polytope $\Delta$.
\end{remark}

We can now use Lemma \ref{symplectomorphismlemma} above to prove Theorem \ref{symplectomorphismtheorem}.

\begin{proof}[Proof of Theorem \ref{symplectomorphismtheorem}]
First, we will assume that $(M_{\ba},\omega_{\ba}^\kappa)$ is symplectomorphic to $(M_{\bb},\omega_{\bb}^\kappa)$.  If this is true, then $H^*(M_{\ba})\cong H^*(M_{\bb})$, which by Proposition \ref{mainlemma}, implies that $M_{\ba}$ and $M_{\bb}$ are isomorphic as symplectic bundles.  Note that the specific choice of symplectomorphism will in general have nothing to do with the choice of isomorphism of symplectic bundles.

Now assume that $M_{\ba}$ is isomorphic to $M_{\bb}$ as a symplectic bundle.  By Proposition \ref{mainlemma}, there exists a vector $C$ so that $\sigma_1(C,\ba+C)=\sigma_1(0,\bb)$, Thus, we have $\sigma_1(\bb)=(r+1)C+\sigma_1(\ba)$ for some $C$. Without loss of generality, we assume $\sigma_1(\ba)\leq\sigma_1(\bb)$.  We show that any vector $\bb$ can be reached from $\ba$ by the following elementary moves.  We will denote by $e_1(\ba)$ the elementary move described by
\[
e_1(\ba)=(a_1+2,a_2+1,\ldots,a_r+1)
\]
and by $e_{i,j}(\ba)$ the elementary move described by
\[
e_{i,j}(\ba)=(a_1,\ldots,a_i-1,\ldots,a_j+1,\ldots,a_r).
\]
Lemma \ref{symplectomorphismlemma} then says that $(M_{\ba},\omega_{\ba}^\kappa)$ is symplectomorphic $(M_{e_1(\ba)},\omega_{e_1(\ba)}^\kappa)$ and to $(M_{e_{i,j}(\ba)},\omega_{e_{i,j}(\ba)})$ for all $i,j$.  Thus, if we can reach $\bb$ from $\ba$ by the elementary moves $e_1$ and $e_{i,j}$, Lemma \ref{symplectomorphismlemma} would give a symplectomorphism from $(M_{\ba},\omega_{\ba}^\kappa)$ to $(M_{\bb},\omega_{\bb}^\kappa)$ as desired.

First, we recall that $\sigma_1(\bb)=(r+1)C+\sigma_1(\ba)$, where by our assumption, $C\geq 0$.  Thus, by repeatedly applying $e_1$, we can get a vector
\[
\ba'=e_1^C(\ba)=(a_1+2C,a_2+C,\ldots,a_r+C)
\]
where $\sigma_1(\ba')=\sigma_1(\bb)$.

Next we can get a vector $\ba^1$ as follows:
\[
\ba^1=e_{1,2}^{a_1+2C-b_1}(\ba')=(b_1,a_1+a_2+3C-b_1,a_2+C,\ldots,a_r+C)
\]
Continuing on by induction, we get vectors $\ba^i$, where
\begin{align*}
\ba^i&=e_{i-1,i}^{a_1+\ldots+a_{i-1}+iC-b_1-\ldots-b_{i-1}}(\ba^{i-1})\\
&=(b_1,\ldots,b_i,a_1+\ldots+a_{i+1}-b_1\ldots-b_i+(i+2)C,a_{i+2}+C,\ldots,a_r+C).
\end{align*}
But then a straightforward computation shows that
\[
\ba^{r-1}=(b_1,\ldots,b_{r-1},a_1+\ldots+a_r+(r+1)C-b_1-\ldots-b_{r-1})=(b_1,\ldots,b_r)=\bb
\]
where $b_r=a_1+\ldots+a_r+(r+1)C-b_1-\ldots-b_{r-1}$ is true because
\[
a_1+\ldots+a_r+(r+1)C=\sigma_1(\ba')=\sigma_1(\bb)=b_1+\ldots+b_r
\]
Thus, we have reached $\bb$ from $\ba$ by using the elementary moves $e_1$ and $e_{i,j}$, as desired.
\end{proof}

Lastly, we will say more about the deformation class of $(M_{\ba}\omega_\kappa^a)$.  In particular, we will prove Lemma \ref{deformationequivalentlemma}, which says that if $M_{\ba}$ and $M_{\bb}$ are both $\bC P^r$ bundles over $\bC P^s$, then they are isomorphic as symplectic bundles if and only if they are deformation equivalent, and furthermore that they are symplectomorphic if $\kappa\gg 0$. This justifies our use of deformation equivalence as the equivalence relation on symplectic manifolds.

\begin{proof}[Proof of Lemma \ref{deformationequivalentlemma}]
First assume that $(M_{\ba},\omega_{\ba}^\kappa)$ is deformation equivalent to $(M_{\bb},\omega_{\bb}^\kappa)$.  Then $M_{\ba}$ and $M_{\bb}$ are diffeomorphic, and in particular, $H^*(M_{\ba})$ is isomorphic to $H^*(M_{\bb})$, so that by Proposition \ref{mainlemma}, $M_{\ba}$ is isomorphic to $M_{\bb}$ as a symplectic bundle.

Now assume that $M_{\ba}$ is isomorphic to $M_{\bb}$ as a symplectic bundle.  This implies that there is a diffeomorphism $\phi:M_{\ba}\rightarrow M_{\bb}$ so that
\[
\phi^*(\omega_{\bb}|_{F_{\phi(x)}})=\omega_{\ba}|_{F_x}
\]
where $F_x$ is the fiber over $x$ and $F_{\phi(x)}$ is the fiber over $\phi(x)$.  In other words, the diffeomorphism $\phi$ preserves the fiberwise symplectic structures of $M_{\ba}$ and $M_{\bb}$, which also implies that $\phi^*([\omega_{\bb}^\kappa])=[\omega_{\ba}^\kappa]$.  We wish to show that $M_{\ba}$ and $M_{\bb}$ are deformation equivalent.  By the above, it suffices to show that there is a family of symplectic forms $\omega_t$ so that $\omega_0=\omega_{\ba}^\kappa$ and $\omega_3=\phi^*(\omega_{\bb}^\kappa)$.  We can produce such a family explicitly.  Namely, if $\pi$ is the map from $M_{\ba}$ to $\bC P^s$ and $\omega_s$ is the standard symplectic form on $\bC P^s$, the deformation $\omega_t$ can be chosen explicitly as
\begin{equation*}
\omega_t = \left\{
\begin{array}{ll}
\omega_{\ba}^\kappa+Kt\pi^*(\omega_s) & \text{if } 0\leq t\leq 1\\
(t-1)\phi^*(\omega_{\bb}^\kappa) +(2-t)\omega_{\ba}^\kappa+K\pi^*(\omega_s)& \text{if } 1\leq t\leq 2\\
\phi^*(\omega_{\bb}^\kappa)+(3-t)K\pi^*(\omega_s) & \text{if } 2\leq t\leq 3.
\end{array} \right.
\end{equation*}
Recall that Lemma \ref{bundlecorrespondencetheorem} says
\[
\omega_{\ba}^{\kappa+K}=\omega_{\ba}^\kappa+\tfrac{K}{s+1}\pi^*(\omega_s)
\]
For $0\leq t\leq 1$, this implies that $\omega_t=\omega_{\ba}^{(\kappa+(s+1)(tK))},$ and hence is non-degenerate.  Also, if $K$ is large enough, then $\omega_t$ for $1\leq t\leq 2$ will all be non-degenerate.  Now, recall that since $\phi$ is an isomorphism of symplectic bundles, $\pi\circ\phi=\pi$, and hence $\pi^*(\omega_s)=\phi^*(\pi^*(\omega_s))$.  Using this and Lemma \ref{bundlecorrespondencetheorem} as above, we have for $2\leq t\leq 3$ that
\begin{align*}
\omega_t&=\phi^*(\omega_{\bb}^\kappa)+(3-t)K\pi^*(\omega_s)\\
&=\phi^*(\omega_{\bb}^\kappa)+(3-t)K\phi^*(\pi^*(\omega_s))\\
&=\phi^*(\omega_{\bb}^\kappa+(3-t)K\pi^*(\omega_s))\\
&=\phi^*(\omega_{\bb}^{\kappa+(s+1)((3-t)K)})
\end{align*}

Lastly, by the above, for any $\lambda>\kappa+K$ with $K$ sufficiently large, $(M_{\ba},\omega_{\ba}^\lambda)$ is isotopic to $(M_{\bb},\omega_{\bb}^\lambda)$ by the linear isotopy $t\omega_{\ba}^\lambda+(1-t)\omega_{\bb}^\lambda$, and hence they are in fact symplectomorphic, as required.
\end{proof}

\section{Proofs of Main Theorems}

We now give the proofs of the main theorems stated in the introduction.  First we will prove Theorem \ref{r<stheorem}, which we recall stated that if $(M^{2n}_{\ba},\omega_{\ba}^\kappa)$ is the $\bC P^r$ bundle over $\bC P^s$ determined by $(\ba;\kappa)$, then $N_n(\ba;\kappa)=1$ when $r<s$.

\begin{proof}[Proof of Theorem \ref{r<stheorem}]Since $r<s$, $\alpha^{r+1}\neq 0$, so by Proposition \ref{mainlemma}, we know $M_{\ba}\sim M_{\bb}$ if and only if
\[
\sigma_i(C,C+a_1,\ldots,C+a_r)=\sigma_i(0,b_1,\ldots,b_r)\quad 1\leq i\leq \min\{r+1,s\}=r+1
\]
If $C=0$, then $(C,C+\ba)=(0,\ba)$.  Therefore, $\sigma_i(0,\ba)=\sigma_i(0,\bb)$ for all $1\leq i\leq r+1$, which implies $\ba=\bb$, as desired.  If $C\neq 0$, then if $\sigma_{r+1}(C,C+\ba)=\sigma_{r+1}(0,\bb)=0$, we must have some $i$ where $C+a_i=0$, so $C=-a_i<0$.  But then there is no way that $\sigma_i(C,C+\ba)=\sigma_i(0,\bb)$, for all $i$ since all $b_i\geq 0$ and $C<0$.  This contradiction finishes the theorem.
\end{proof}

We now focus on proving the theorems stated for the $r\geq s$ case.  Before we do that however, we give an example of a vector $\ba$ and constant $\kappa$ so that $N(\ba;\kappa)>1$.  We first note that by Proposition \ref{mainlemma} and Lemma \ref{toricstructurelemma}, we must only produce two vectors $\ba=(a_1,\ldots, a_r)$ and $\bb=(b_1,\ldots,b_r)$ and a number $C$ so that
\[
\sigma_i(C,C+\ba)=\sigma_i(0,\bb),~1\leq i\leq\min\{r+1,s\}=s
\]
with $\ba\neq \bb$.  Indeed, then by Lemma \ref{toricstructurelemma}, since $\ba\neq\bb$ they represent different toric structures but by Proposition \ref{mainlemma} and Lemma \ref{deformationequivalentlemma}, we know that the underlying manifolds will be deformation equivalent. But $r\geq s$, so $\min(r+1,s)=s$, which does not force $(C,C+\ba)=(0,\bb)$.  We have the following explicit example.

\begin{example}\label{r=3s=2example}Let $\ba=(1,4,4)$, $\bb=(2,2,5)$ describe $\bC P^3$ bundles over $\bC P^2$.  Since $\sigma_1(\ba)=\sigma_1(\bb)$, $K_{\ba}=K_{\bb}$, where we recall $K_{\ba}$ is the number so that $\Delta_{\ba}^\kappa$ is a bundle for all $\kappa>K_{\ba}$.  Thus, as we increase $\kappa$, the two toric structures $(M_{\ba},\omega_{\ba}^\kappa,T_{\ba})$ and $(M_{\bb},\omega_{\bb}^\kappa,T_{\bb})$ will both appear at the same time, so that the corresponding jump in $N_5(\ba;\kappa)$, which occurs at $\kappa=7$, will be a jump of size $2$.  Also, a fairly simple check will show that there is no other choice of vector $\bc$ such that $M_{\bc}$ is equivalent to $M_{\ba}$.  More specifically, by Proposition \ref{mainlemma}, the only options would be vectors $\bc$ that had $\sigma_1(\bc)=5,1$, or $\sigma_1(\bc)\geq 9$, corresponding to $C=-1$ or $C=-2$, or $C\geq 0$.  The $C=-1$ and $C=-2$ cases can easily be checked not to work by hand.  If $C=0$, $(1,4,4)$ and $(2,2,5)$ are the only solutions, as a simple computation shows.

Now assume $C=1$.  If we take the vector $(1,4,4)$ and look for more examples with $C=1$, we must compare the vector $(1,2,5,5)$ to an arbitrary vector $(0,d_1,d_2,d_3)$.  But $\sigma_1(1,2,5,5)=13$ and $\sigma_2(1,2,5,5)=57$, while the biggest that $\sigma_2(0,d_1,d_2,d_3)$ could be with $\sigma_1(0,d_1,d_2,d_3)=13$ is when $(0,d_1,d_2,d_3)=(0,4,4,5)$, which has $\sigma_2(0,4,4,5)=56$.  Note that $(0,4,4,5)$ is indeed the biggest because it is the vector which is closest to having all terms equal, which an exercise in calculus will confirm is the biggest. That there are no examples with $C\geq 2$ follows directly from Lemma \ref{sigma2inductionlemma}.
Hence, for the above choice of $\ba=(1,4,4)$, we have
\begin{equation*}
N_5(\ba;\kappa)= \left\{
\begin{array}{ll}
0 & \text{if } \kappa\leq 7\\
2 & \text{if } \kappa>7.
\end{array} \right.
\end{equation*}
\end{example}

We will now go back and prove the various theorems we stated for the case $r\geq s$, starting with Theorem \ref{monotoneuniquenesstheorem}, which says that $N_n(\ba;\kappa)=1$ if $\kappa\leq 1$.
\begin{proof}[Proof of Theorem \ref{monotoneuniquenesstheorem}]
If $\ba$ defines a toric structure with $\kappa\leq 1$, then we know that $\kappa>-s+\sigma_1(\ba)=K_{\ba}$, so in particular,
\[
0<\sigma_1(\ba)<s+\kappa\leq s+1
\]
since $\kappa\leq 1$.  Proposition \ref{mainlemma} implies that if $\ba$ determines a bundle with a non-unique structure, then $r\geq s$ and there is a vector $b=(b_1,\ldots,b_r)$ and a number $C$ so that
\[
\sigma_i(C,C+a_1,\ldots,C+a_r)=\sigma_i(0,b_1,\ldots,b_r)\quad 1\leq i\leq s
\]
In our case, $\ba$ and $\bb$ satisfy $a_i\geq 0$ and $b_i\geq 0$, so that $\sigma_1(\ba)>0$ and $\sigma_1(\bb)>0$.  If they are both to be valid toric structures with $\kappa\leq 1$, they must also satisfy $\sigma_1(\ba)\leq s$ and $\sigma_1(\bb)\leq s$, as we saw above.  But $\sigma_1(C,C+\ba)=(r+1)C+\sigma_1(\ba)$, so putting this all together, we see that if $\sigma_1(C,C+\ba)=\sigma_1(0,\bb)=\sigma_1(\bb)$, then $C=0$ and $\sigma_1(\ba)=\sigma_1(\bb)\leq s$.

Now assume $\ba=(a_1,\ldots,a_r)$ with $a_i\geq 0$ and $\sigma_1(\ba)=k$ with $1\leq k\leq s$.  Then since $a_i\in\bZ$, the vector $\ba$ has at most $s$ non-zero terms. Therefore, for any two vectors $\ba$ and $\bb$ as above, we must have $\sigma_i(\ba)=\sigma_i(\bb)=0$ for all $i>s$.  Therefore, if $\sigma_i(\ba)=\sigma_i(\bb)$ for all $1\leq i\leq s$, we actually have $\sigma_i(\ba)=\sigma_i(\bb)$ for all $i$, which means that $\ba=\bb$ up to reordering, as required.
\end{proof}

We next prove Theorem \ref{stepfunctiontheorem}, which is a simple consequences of the above machinery.  Recall that Theorem \ref{stepfunctiontheorem} says first that the function $N_n(\ba;\kappa)$ is a step function which can only have jumps at the values $K_M+l(r+1)$, and second that if $r=s$, then these potential jumps are all of size $1$.

\begin{proof}[Proof of Theorem \ref{stepfunctiontheorem}]We first prove statement $(1)$.  Recall that $K_M\in\bZ$ is the largest number so that $N_n(\ba;K_M)=0$.  Recall also that $K_M$ need not equal $K_{\ba}$ for every possible $\ba$.  However, by the definition of $K_M$, there is always some vector $\bb$ so that $K_{\bb}=K_M$.  For convenience sake, we will assume that $K_{\ba}=K_M$. We know from Proposition \ref{mainlemma} that if there is another inequivalent toric structure on $M_{\ba}$, there is a vector $\bb$ so that $\ba\neq \bb$ and an integer $C$ so that
\[
\sigma_i(C,C+a_1,\ldots,C+a_r)=\sigma_i(0,b_1,\ldots,b_r)\quad 1\leq i\leq s<r+1
\]
In particular, we know that $C\geq 0$, since any $\bb$ determining a toric structure on $M_{\ba}$ must have $K_{\bb}\geq K_M=K_{\ba}$ which implies $\sigma_1(\bb)\geq\sigma_1(\ba)$.  Thus, $\sigma_1(\bb)=\sigma_1(\ba)+C(r+1)$ for some integer $C\geq 0$ and the value of $N_n(\ba;\kappa)$ can only jump at the values of $\kappa$ where
\[
\kappa=K_{\bb}=-s+\sigma_1(\bb)=-s+\sigma_1(\ba)+C(r+1)=K_{\ba}+l(r+1)=K_M+l(r+1)
\]
where $l=C\geq 0$, which finishes the proof of statement $(1)$.

We now prove statement $(2)$.  Assume that for some $\kappa$, there is a jump of size $2$ or more.  Then there exist two vectors $\ba\neq\bb$ with $K_{\ba}=K_{\bb}$ and a constant $C$ so that
\[
\sigma_i(C,C+a_1,\ldots,C+a_r)=\sigma_i(0,b_1,\ldots,b_r)\quad 1\leq i\leq s<r+1
\]
But $K_{\ba}=K_{\bb}$ implies $\sigma_1(\ba)=\sigma_1(\bb)$, which implies that $C=0$, which obviously implies
\[
\sigma_i(\ba)=\sigma_i(0,\ba)=\sigma_i(0,\bb)=\sigma_i(\bb)
\]
for all $1\leq i\leq s=r$, which implies that $\ba=\bb$ up to reordering.  This contradiction establishes statement $(2)$.
\end{proof}

Next, we prove Theorem \ref{finitenesstheorem}.  Recall that this theorem said that
\[
\Bigl(\kappa_1,\kappa_2>(r+1-\tfrac{1}{r})\sigma_1(\ba)-s\Bigr)\Longrightarrow \bigl(N_n(\ba;\kappa_1)=N_n(\ba;\kappa_2)\bigr).
\]
Before we prove this however, we prove a technical lemma which is the main force behind the theorem.

\begin{lemma}\label{finitenesslemma}Fix an $r\geq s\geq 2$.  Assume we have non-negative integer vectors $\ba=(a_1,\ldots,a_r)$ and $\bb=(b_1,\ldots,b_r)$ as before, and a real number $C$ so that
\[
\sigma_i(C,C+\ba)=\sigma_i(0,\bb),\quad\forall 1\leq i\leq s<r+1
\]
Then
\[
-\tfrac{1}{r+1}\sigma_1(\ba)\leq C\leq \tfrac{r-1}{r}\sigma_1(\ba)
\]
\end{lemma}
\begin{proof}First, notice that if $C<-\tfrac{1}{r+1}\sigma_1(\ba)$, then
\[
\sigma_1(\bb)=(r+1)C+\sigma_1(\ba)<0
\]
which contradicts the fact that $\bb$ is a positive integer vector.  Thus, we must have $C\geq -\tfrac{1}{r+1}\sigma_1(\ba)$.

It remains to show that $C\leq\tfrac{r-1}{r}$.  Since $s\geq 2$, it suffices to show that if $C>\tfrac{r-1}{r}$, then any non-negative integer vector $\bb$ with $\sigma_1(0,\bb)=\sigma_1(C,C+\ba)$ satisfies
\[
\sigma_2(0,\bb)<\sigma_2(C,C+\ba).
\]
Indeed, for these values of $C$, there cannot exist a vector $\bb$ with $\sigma_i(0,\bb)=\sigma_i(C,C+\ba)$ for all $i$.   To see this, we will consider the two polynomials
\[
P_{\ba}(C):=\sigma_2(C,C+\ba),\quad P_{\bb}(C):=\sigma_2(0,\tfrac{(r+1)C+\sigma_1(\ba)}{r},\ldots,\tfrac{(r+1)C+\sigma_1(\ba)}{r})
\]
Notice that any vector $\bb$ with $\sigma_1(0,\bb)=\sigma_1(C,C+\ba)$, has $\sigma_2(\bb)<P_{\bb}(C)$ as a consequence of basic calculus. Indeed, the quantity $\sigma_2(0,b_1,\ldots,b_r)$ is maximized by $b_1=\ldots=b_r$, and the inequality follows from this fact.  Thus, to prove the theorem, it only remains to show that $P_{\ba}(C)-P_{\bb}(C)>0$ for all $C>\tfrac{r-1}{r}\sigma_1(\ba)$.  We will do this by showing that
\[
(P_{\ba}-P_{\bb})(\tfrac{r-1}{r}\sigma_1(\ba))>0,\quad (P_{\ba}-P_{\bb})'(\tfrac{r-1}{r}\sigma_1(\ba))\geq 0
\]
Then since $P_{\ba}-P_{\bb}$ is a degree $2$ polynomial, the desired result will follow.  We show this by explicitly computing both terms.

First, we see that
\[
P_{\ba}(C)=\sigma_2(C,C+\ba)=\tbinom{r+1}{2}C^2+\tbinom{r}{1}\sigma_1(\ba)C+\sigma_2(\ba)=\tfrac{(r+1)(r)}{2}C^2+r\sigma_1(\ba)C+\sigma_2(\ba).
\]

Next, after some rearranging, we see that
\begin{align*}
P_{\bb}(C)=\sigma_2(0,\tfrac{(r+1)C+\sigma_1(\ba)}{r},\ldots,\tfrac{(r+1)C+\sigma_1(\ba)}{r})&=\tbinom{r}{2}\left(\tfrac{r+1}{r}\right)^2C^2+2\tbinom{r}{2}\tfrac{r+1}{r^2}\sigma_1(\ba)C+\tbinom{r}{2}\tfrac{\sigma_1(\ba)^2}{r^2}\\
&=\tfrac{r(r-1)(r+1)^2}{2r^2}C^2+\tfrac{r(r-1)(r+1)\sigma_1(\ba)}{r^2}C+\tfrac{r(r-1)\sigma_1(\ba)^2}{2r^2}\\
&=\left(\tfrac{r^2-1}{r^2}\right)\left(\tfrac{(r+1)(r)}{2}C^2+r\sigma_1(\ba)C\right)+\tfrac{r-1}{2r}\sigma_1(\ba)^2.
\end{align*}
A simple computation gives
\[
(P_{\ba}-P_{\bb})(C)=\tfrac{r+1}{2r}C^2+\tfrac{1}{r}\sigma_1(\ba)C+\sigma_2(\ba)-\tfrac{r-1}{2r}\sigma_1(\ba)^2.
\]
Also, taking the derivative of this, we get that
\[
(P_{\ba}-P_{\bb})'(C)=\tfrac{r+1}{r}C+\tfrac{1}{r}\sigma_1(\ba).
\]
We then have the following computation:
\begin{align*}
(P_{\ba}-P_{\bb})(\tfrac{r-1}{r}\sigma_1(\ba))&=\tfrac{(r+1)(r-1)^2}{2r^3}(\sigma_1(\ba))^2+\tfrac{2r^2-2r}{2r^3}(\sigma_1(\ba))^2+\sigma_2(\ba)-\tfrac{r^3-r^2}{2r^3}\sigma_1(\ba)^2\\
&=\tfrac{2r^2-3r+1}{2r^3}\sigma_1(\ba)^2+\sigma_2(\ba)=\tfrac{(2r-1)(r-1)}{2r^3}\sigma_1(\ba)^2+\sigma_2(\ba)>0.
\end{align*}
where the last inequality follows since $a_i\geq 0$ and some $a_i\neq 0$ and $r\geq s>1$, which implies that $(2r-1)>0$ and $r-1>0$.  We also have
\[
(P_{\ba}-P_{\bb})'(\tfrac{r-1}{r}\sigma_1(\ba))=(\tfrac{r+1}{r}\tfrac{r-1}{r}+\tfrac{1}{r})\sigma_1(\ba)\geq 0.
\]
This computation completes the proof.
\end{proof}

The following similar lemma is useful in applications.

\begin{lemma}\label{sigma2inductionlemma}Fix an integer $C\geq 1$ and a non-negative integer vector $\ba$.  Consider the inequalities
\begin{equation*}
\sigma_2(0,\bb)\leq\sigma_2(C,C+\ba)
\tag{$*_0$}\end{equation*}
\begin{equation*}
\sigma_2(0,\bb)<\sigma_2(C+n,C+n+\ba)
\tag{$*_n$}
\end{equation*}
where $n\geq 1$ and in $(*_k)$ with $0\leq k\leq n$, $\bb$ ranges over all integer vectors with $\sigma_1(0,\bb)=\sigma_1(C+k,C+k+\ba)$.  Then
\[
(*_0)\Longrightarrow (*_n)~\forall~n\geq 1.
\]
\end{lemma}
\begin{proof}
An obvious induction shows that it suffices to prove the theorem in the case $n=1$.  Write $(r+1)C+\sigma_1(a)=kr+l$ for some integers $k,l$, where $k>0$ since $C\geq 1$ and $0\leq l<r$.  We will call $\ba'=(a'_0,\ldots,a'_r)=(C,C+\ba)$.  Then the integer vector $\bb$ with $\sigma_1(\bb)=\sigma_1(C,C+\ba)$ with largest value of $\sigma_2(0,\bb)$ is $\bb=(k,\ldots,k,k+1,\ldots,k+1)$ with exactly $r-l$ entries equal to $k$ and $l$ entries equal to $k+1$.  By assumption, we know $\sigma_2(0,\bb)\leq\sigma_2(\ba')$.  Now, consider $(C+1,C+1+\ba)$.  Then $\sigma_1(C+1,C+1+\ba)=(k+1)r+l+1$ and the vector $\bb'$ with this $\sigma_1$ and the largest $\sigma_2$ is now $\bb'=(k+1,\ldots,k+1,k+2,\ldots,k+2)$ where here there are $l+1$ entries equal to $k+2$ and $r-l-1$ entries equal to $k+1$.  Then we have $(C+1,C+1+\ba)=(a'_0+1,\ldots,a'_r+1)$ and $\bb'=(b_1+1,\ldots,b_{r-l}+2,\ldots,b_r+1)$, since $b_{r-l}=k$ while $b_{r-l}'=k+2$.  A simple computation shows that
\[
\sigma_2(C+1,C+1+\ba)=\sigma_2(\ba')+r\sigma_1(\ba')+\tbinom{r+1}{2}
\]
Another simple computation shows that
\begin{align*}
\sigma_2(\bb')&=\sigma_2(\bb)+r\sigma_1(\bb)-b_{r-l}+\tbinom{r-1}{2}+2r-2=\sigma_2(\bb)+r\sigma_1(\ba')-b_{r-l}+\tfrac{r^2-3r+2+4r-4}{2}\\
&=\sigma_2(\bb)+r\sigma_1(\ba')-k+\tfrac{r^2+r}{2}-1=\sigma_2(\bb)+r\sigma_1(\ba')+\tbinom{r+1}{2}+(-1-k)\\
&<\sigma_2(\ba')+r\sigma_1(\ba')+\tbinom{r+1}{2}=\sigma_2(C+1,C+1+\ba)
\end{align*}
where the last inequality is true because $\sigma_2(\bb)\leq\sigma_2(\ba')$ and $-1-k<0$.  Thus, we have the desired result.
\end{proof}

\begin{proof}[Proof of Theorem \ref{finitenesstheorem}]
First, by Theorem \ref{r<stheorem}, it suffices to consider $r\geq s>1$.  We will show that for any non-negative integer vector $\ba$, $N_n(\ba;\kappa_1)-N_n(\ba;\kappa_2)=0$ when $\kappa_1,\kappa_2>(r+1-\tfrac{1}{r})\sigma_1(\ba)-s$.  Without loss of generality, assume that $\kappa_1>\kappa_2$.  Now, if the result were false, we would have $N_n(\ba;\kappa_1)-N_n(\ba;\kappa_2)>0$, which would mean there was some vector $\bb$ with $\kappa_2<K_{\bb}\leq \kappa_1$ and a corresponding number $C$ so that
\[
\sigma_i(C,C+a_1,\ldots,C+a_r)=\sigma_i(0,b_1,\ldots,b_r)\quad 1\leq i\leq s<r+1
\]
Also, since $K_{\bb}>\kappa_2>(r+1-\tfrac{1}{r})\sigma_1(\ba)-s$ and $K_{\bb}=-s+\sigma_1(\bb)$, we know that
\begin{align*}
\sigma_1(\bb)=(r+1)C+\sigma_1(a)>(r+1-\tfrac{1}{r})\sigma_1(\ba)&\Longrightarrow(r+1)C>\tfrac{r^2-1}{r}\sigma_1(\ba)\\
&\Longrightarrow C>\tfrac{r^2-1}{r(r+1)}\sigma_1(\ba)=\tfrac{r-1}{r}\sigma_1(\ba)
\end{align*}
But this is impossible by Lemma \ref{finitenesslemma} above.
\end{proof}

We will now give a proof of Theorem \ref{unboundedtheorems=2}, which says that for any $r\geq s=2$,
\[
\sup_{\ba} N_{r+s}(\ba;\infty)=\infty
\]

\begin{proof}First we show that the above can be reduced to the case $r=s$.  A simple computation shows that for any vectors $\ba$ and $\bb$ and any real number $\kappa$,
\[
\sigma_i(\ba)=\sigma_i(\bb)~\forall 1\leq i\leq s,\Longrightarrow\sigma_i(\ba+\kappa)=\sigma_i(\bb+\kappa)~\forall 1\leq i\leq s
\]
Next, consider a vector $\ba=(a_1,\ldots,a_r)$ with $N_{2r}(\ba;\infty)=k$.  Then there exist vectors $\bb_1,\ldots,\bb_{k-1}$ with corresponding constants $C_{1},\ldots,C_{k-1}$ so that
\[
\sigma_i(C_j,\ba+C_j)=\sigma_i(0,\bb_j)~\forall 1\leq i\leq s,~1\leq j\leq k-1
\]
Now consider some $r>s$.  Then $r=s+l$ for some $l\geq 1$.  We can define the following vectors:
\[
\bb_j^l=(0,\ldots,0,C_{k-1}-C_j,C_{k-1}-C_j+\bb_j)
\]
where this vector has $l-1$ $0s$.  Then, the above computation shows that
\[
\sigma_i(0,\bb_j^l)=\sigma_i(0,\ldots,0,C_{k-1},C_{k-1}+\ba)~\forall 1\leq i\leq s,~1\leq j\leq k-1,~l\geq 1
\]
This shows that if the theorem holds for $r=s$, then it holds for any $r>s$.

We now consider the case where $r=s=2$.  We will show that for any $k$, there exists a vector $\ba_k=(a_k,b_k)$ with $N_4(\ba_k;\infty)=k$.

To see this, notice that any vector $\ba=(a,b)$ and any vector $\bb$ with $\sigma_1(0,\bb)=(C,C+a,C+b)$ can be written as $\bb=(C+a+x,2C+b-x)$ for some integer $x$.  Then
\[
\sigma_2(0,C+a+x,2C+b-x)=\sigma_2(C,C+a,C+b)\Longleftrightarrow bx-ax+Cx-x^2=bC+C^2
\]
We will look for solutions of the special form $C=\lambda x$.  Now, substituting for C and solving for $x$ gives
\[
x=\tfrac{b-a-b\lambda}{\lambda^2-\lambda+1}
\]
Thus, any choice of $a,b$ and $\lambda$ such that $x$ and $C$ are both integers will result in a vector $\bb$ with $\bb\neq \ba$, but $M_{\ba}\sim M_{\bb}$.  We consider the family where $\lambda=\tfrac{1}{n}$.  Substituting for $\lambda$, we have the following computation
\[
x=\tfrac{b-a-\tfrac{b}{n}}{\tfrac{1}{n^2}-\tfrac{1}{n}+1}=\tfrac{\tfrac{n-1}{n}b-a}{\tfrac{n^2-n+1}{n^2}}=\tfrac{n}{n^2-n+1}((n-1)b-na),\quad C=\tfrac{1}{n^2-n+1}((n-1)b-na)
\]
Thus, if we can find a pair of integers $a,b$ with $(n-1)b-na\equiv 0~\mod (n^2-n+1)$, then $x$ and $C$ will be integers as desired.  More specifically, if for each $k$ we can find integers $a_k,b_k$ and $k-1$ integers $n_1,\ldots,n_{k-1}$ with
\begin{align*}
(n_i-1)b_k-n_ia_k&\equiv 0 \mod (n_i^2-n_i+1)\Longleftrightarrow \\
&-b_k \equiv n_i(a_k-b_k) \mod (n_i^2-n_i+1), \forall 1\leq i\leq k-1
\end{align*}
then $\ba_k=(a_k,b_k$ we would have $N_4(\ba_k;\infty)\geq k$ for each $k$, as desired.  We will solve these equations using the Chinese Remainder Theorem.  More specifically, we restrict our attention to vectors of the form $(K,c_k+K)$, for a fixed integer $c_k$ which fixes the quantity $a_k-b_k=-c_k$. Then, plugging in and simplifying, we have reduced the problem to picking an integer $K$ so that
\[
K\equiv n_i(c_k-1) \mod n_i^2-n_i+1
\]
for some collection of integers $n_i$.
The Chinese Remainder Thereom then says that this system of equations will have a solution provided that the collection of integers $N_i:=n_i^2-n_i+1$ can be chosen to be relatively prime.  Thus, to complete the proof, we only need to produce a sequence $N_i=n_i^2-n_i+1$ such that $\gcd(N_i,N_j)=1$ for all $i,j$.  We will produce this sequence by induction.  In particular, we will produce a sequence $N_n$ such that if $i<j$, all prime factors of $N_i$ are less than all prime factors of $N_j$.  Such a sequence would obviously have $\gcd(N_i,N_j)=1$.

First, let $n_1=2$, so that $N_1=3$.  Now, assume that we have such integers $N_1,\ldots,N_{k-1}$ with corresponding integers $n_1,\ldots,n_{k-1}$ so that $N_i=n_i^2-n_i+1$ and such that if $i<j$, all prime factors of $N_i$ are less then all prime factors of $N_j$.  Let $p_k$ be the largest prime number dividing $N_{k-1}$.  Note that by our assumption that $N_1=3$, such a number $p_k$ will always exist for any $k$.  Now, let $n_k=p_k!$ and let $N_k=n_k^2-n_k+1$.  Then if $q$ is any prime such that $q\leq p_k$, then by construction we have $N_k\equiv 1 \mod q$, and hence the only primes dividing $N_k$ are bigger than $p_k$, as desired.  This computation finishes the construction of the sequence $N_n$, and hence finishes the proof of the theorem.
\end{proof}

The above techniques show that to prove Conjecture \ref{unboundedconjecture} for any $r\geq s\geq 2$, it is enough to check it for any $r=s$.  However, if $r=s\geq3$, the equations involved are much more complicated than for the $s=2$ case above, and it is not clear how to show directly that $\sup_{\ba}(N_{2r}(\ba;\infty))=\infty$.

We conclude the paper with a few interesting examples which explore Theorem \ref{finitenesstheorem} in the Fano case.  First, we explore the case $r=s=2$.

\begin{example}\label{fanoexample1}We claim that if $r=s=2$ and $(M_{\ba},\omega_{\ba}^\kappa)$ is Fano, then $N_4(\ba;\infty)=1$.  We recall from Remark \ref{Fanodefinition} that if $\ba$ is a Fano vector, we must have $K_a<1$.  However, since $s=2$, a simple computation shows that this implies that we must have $\sigma_1(\ba)\leq 2$, which gives us the $4$ cases $\ba=(0,0)$, $\ba=(0,1)$, $\ba=(1,1)$, and $\ba=(0,2)$.  We recall that Proposition \ref{producttheorem}, proven in \cite{D}, implies that if $\ba=(0,0)$, then $N_4(\ba;\infty)=1$.  Thus, we only need to consider $(0,1)$, $(1,1)$, and $(0,2)$.  The cases $(0,1)$ and $(1,1)$ are special cases of Example \ref{fanoexample2} below, and for $\ba=(0,1)$ or $\ba=(1,1)$, we get $N_4(\ba;\infty)=1$ as well.  Thus, it suffices to check that $N_4((0,2);\infty)=1$.

Since $r=s=2$, by Lemma \ref{finitenesslemma} it suffices to show that there is no vector $\bb$ and integer $C$ such that
\[
\sigma_i(0,\bb)=\sigma_i(C,2+C),\quad i=1,2,\quad -\tfrac{2}{3}\leq C\leq 1
\]
Since $r=s=2$, Theorem \ref{stepfunctiontheorem} says that there are no other examples with $C=0$.  Thus, it suffices to check that there is no examples with $C=1$.  Note that $\sigma_1(1,1,3)=5$ and $\sigma_2(1,1,3)=7$.  Thus, we must show there is no vector $\bb=(b_1,b_2)$ so that $\sigma_1(\bb)=5$ and $\sigma_2(\bb)=7$.  However, the vector $\bb$ with $\sigma_1(\bb)=5$ and largest $\sigma_2(\bb)$ is the vector $\bb=(2.5,2.5)$, which has $\sigma_2(\bb)=6.25<7$.  Thus, there is no $\bb$ with $\sigma_1(\bb)=5$ and $\sigma_2(\bb)=7$, so that $N_4((0,2);\infty)=1$.  Thus, if $r=s=2$ and $(M_{\ba},\omega_{\ba}^\kappa)$ is Fano, then $N_4(\ba;\infty)=1$.
\end{example}

Next, we examine some higher dimensional Fano examples.

\begin{example}\label{fanoexample2}Consider vectors of the form $\ba=(0,\ldots,0,1,\ldots,1)$ where $\sigma_1(\ba)=k$.  We will show that for such vectors, $N_{r+s}(\ba;\kappa)=1$ for all $\kappa>k-s$, and in particular, $N_{r+s}(\ba;\infty)=1$.

If this is false, then there exists some $\bb$ and $C$ so that $\sigma_i(C,C+\ba)=\sigma_i(0,\bb)$ where $i\geq 2$.  We will show that for our specific choice of $\ba$, this does not happen.

First, notice cannot choose $C<0$ because then $\sigma_1(b)<0$.  Second, we cannot choose $C=0$.  Indeed, in that case, any non-negative integer vector $\bb$ with $\sigma_1(\bb)=k$ has $\sigma_2(\bb)\leq\sigma_2(\ba)$, with equality only if $\ba=\bb$.  Thus, we must only show that we cannot choose $C>0$.  This will use Lemma \ref{sigma2inductionlemma}.

First, assume $C=1$.  In this case, we have $\sigma_1(1,1+\ba)=r+k+1$, so that the non-negative integer vector $\bb=(0,b_1,\ldots,b_r)$ with the largest value of $\sigma_2$ will be the vector $\bb=(0,1,\ldots,1,2,\ldots,2)$ with exactly $r-k-1$ entries equal to $1$ and exactly $k+1$ entries equal to $2$.  On the other hand, $(1,1+\ba)=(1,\ldots,1,2,\ldots,2)$ which has exactly $r-k+1$ entries equal to $1$ and exactly $k$ entries equal to $2$.  Thus, we see that
\[
(1,1+\ba)=(b_0+1,b_1,\ldots,b_{r-k-1},b_{r-k}-1,b_{r-k+1},\ldots,b_r).
\]
Thus, using the above substitution, a simple computation shows that
\begin{align*}
\sigma_2(1,1+\ba)&=\sigma_2(\bb)+(\sigma_1(\bb)-b_0)-(\sigma_1(\bb)-b_{r-k})-1\\
&=\sigma_2(\bb)+b_{r-k}-b_0-1=\sigma_2(\bb)+2-1=\sigma_2(\bb)+1
\end{align*}
so that $\sigma_2(1,1+\ba)$ is bigger then $\sigma_2(\bb)$, which means that there is no allowable choice of $\bb$ when $C=1$. Then, Lemma \ref{sigma2inductionlemma} shows that there are no examples with $C>1$, so that $N_n(\ba;\infty)=1$, as claimed.
\end{example}

The next example shows that the general Fano case is not as nice, and in fact there are examples of Fano toric manifolds which have more than one toric structure.

\begin{example}\label{fanoexample3}Let $\ba$ be the vector $\ba=(0,\ldots,0,2)$ with $r\geq 3$.  Note that here $r\geq 3$ is necessary, as is seen in Example \ref{fanoexample1}.  Notice that this vector is Fano with $s=2$, since $K_{\ba}(2)=\sigma_1(\ba)-2=2-2=0<1$.  On the other hand, by choosing $C=1$, we can consider the vector $(1,1+\ba)=(1,\ldots,1,3)$ with exactly $r+1\geq 3$ entries equal to $1$.  Then, consider the vector $\bb=(0,1,\ldots,1,2,2,2)$ with exactly $3$ entries equal to $2$.  First, we see that
\[
\sigma_1(\bb)=r+3=\sigma_1(1,1+\ba).
\]
However, we also have
\[
(1,1+\ba)=(b_0+1,b_1,\ldots,b_{r-3},b_{r-2}-1,b_{r-1}-1,b_r+1).
\]
Note that for the above to make sense, we must have $b_{r-2}\neq b_0$ which is true since $r\geq 3$.  So, using the above substitution, an easy computation shows that
\begin{align*}
\sigma_2(1,1+\ba)&=\sigma_2(b)+(\sigma_1(\bb)-b_0)-(\sigma_1(\bb)-b_{r-2})-(\sigma_1(\bb)-b_{r-1})+(\sigma_1(\bb)-b_r)\\
&-1-1+1+1-1-1=\sigma_2(\bb)-b_0+b_{r-2}+b_{r-1}-b_r-2=\sigma_2(\bb)
\end{align*}
so that $\sigma_i(\bb)=\sigma_i(1,1+\ba)$ for $i=1,2$.  Also, an easy application of Lemma \ref{finitenesslemma} says that $C\geq 2$ is not possible.  Indeed, by Lemma \ref{finitenesslemma}, we know that if $C$ is to support a vector $\bb$ with the desired properties, then
\[
C<\tfrac{r-1}{r}\sigma_1(\ba)=\tfrac{r-1}{r}2<2.
\]
This computation finishes the proof of the following:
\begin{equation*}
N_{r+2}(\ba;\kappa)= \left\{
\begin{array}{ll}
0 & \text{if } \kappa\leq 0\\
1 & \text{if } 0<\kappa\leq r+1\\
2 & \text{if } r+1<\kappa.
\end{array} \right.
\end{equation*}
\end{example}

\end{document}